\documentclass[a4paper,11pt, oneside]{amsart}

\usepackage{amsmath,amssymb,amsthm,enumitem,hyperref, graphicx,mathtools,comment, tikz-cd}
\usepackage[left=2.5cm,right=2.5cm,top=2.6cm,bottom=2.6cm]{geometry}
\setenumerate{label=(\roman*)}
\usetikzlibrary{cd}
\usepackage[capitalize]{cleveref}
\usepackage{stmaryrd}
\usepackage{tikz-cd}
\usepackage{graphicx}
\usepackage{adjustbox}
\usepackage{seqsplit}
\usepackage{thmtools}

\DeclareMathOperator{\Out}{Out}

\DeclareMathOperator{\Id}{Id}

\DeclareMathOperator{\Hom}{Hom}

\DeclareMathOperator{\rank}{rank}
\DeclareMathOperator{\Mat}{Mat}

\DeclareMathOperator{\Tor}{Tor}
\DeclareMathOperator{\tor}{tor}

\DeclareMathOperator{\gr}{\operatorname{th}}
\DeclareMathOperator{\cd}{cd}
\DeclareMathOperator{\coker}{coker}
\DeclareMathOperator{\im}{Im}

\DeclareMathOperator{\pr}{pr}
\DeclareMathOperator{\rk}{rk}
\DeclareMathOperator{\SMRF}{SModRF}
\DeclareMathOperator{\trace}{tr}
\DeclareMathOperator{\rr}{rr}

\newcommand{\fbyz}{(finitely generated free)-by-cyclic }

\DeclareMathOperator{\GL}{GL}
\newcommand{\ZZ}{\mathbb{Z}}
\newcommand{\RR}{\mathbb{R}}
\newcommand{\Z}{\mathbb{Z}}
\newcommand{\CC}{\mathbb{C}}

\newcommand{\QQ}{\mathbb{Q}}
\newcommand{\FF}{\mathbb{F}}

\newcommand{\FP}{\mathrm{FP}}

\newcommand{\UO}{\mathcal{U}}
\newcommand{\VN}{\mathcal{N}}

\newtheorem{problem}{Problem}

\newtheorem{theorem}{Theorem}[section]
\newtheorem{lemma}[theorem]{Lemma}
\newtheorem{proposition}[theorem]{Proposition}
\newtheorem{claim}[theorem]{Claim}
\newtheorem{corollary}[theorem]{Corollary}

\theoremstyle{definition}
\newtheorem{definition}{Definition}[section]

\theoremstyle{remark}
\newtheorem{remark}{Remark}

\usepackage{tikz}
\usetikzlibrary{arrows,quotes}
\tikzstyle{blackNode}=[fill=black, draw=black, shape=circle]

\newcounter{comments}

\title{Thurston norm, polytopes and splitting complexity} 
\author{Andrei Jaikin-Zapirain}
 \address[A. Jaikin-Zapirain]{Instituto de Ciencias Matem\'aticas, CSIC-UAM-UC3M-UCM}
 \email{andrei.jaikin@icmat.es}

\author{Monika Kudlinska}
\address[M.\ Kudlinska]{Emmanuel College, University of Cambridge, St Andrew's Street, Cambridge CB2~3AP, United Kingdom}
\email{mak74@cam.ac.uk}

\author{Pablo S\'{a}nchez-Peralta}
\address[P. S\'{a}nchez-Peralta]{Emmanuel College, University of Cambridge, St Andrew's Street, Cambridge CB2~3AP, United Kingdom}
\email{ps2089@cam.ac.uk}

\begin{document}

\begin{abstract} 
Let $G$ be a finitely generated torsion-free group satisfying the Strong Atiyah
Conjecture. Assuming that \(b_1^{(2)}(G)=0\), we associate to each surjective integral
character \(\phi\) the quantity \(b_1^{(2)}(\ker \phi)\). We prove that the
resulting function extends to a seminorm on \(H^1(G;\mathbb{R})\) and is the
thickness function of a polytope in \(H_1(G;\mathbb{R})\).

More generally, assuming that $G$ is of type
\(\mathrm{FP}_k(\mathbb{Q})\) and  has
\(b_k^{(2)}(G)=0\),  
we prove that the function assigning
\(b_k^{(2)}(\ker \phi)\) to each surjective integral character \(\phi\) is
likewise the thickness function of a polytope. This confirms a conjecture of
Friedl, Lück, and Tillmann.

For torsion-free virtually (finitely generated free)-by-cyclic  groups, we relate the resulting invariant to the
\(L^2\)-Betti complexity of a character, thereby confirming a conjecture of
Gardam and Kielak. As an application, we establish the existence of an
algorithm for computing the Bieri--Neumann--Strebel invariant of
free-by-cyclic groups and discuss connections with the isomorphism problem
for such groups.
\end{abstract}
 
\maketitle

\section{Introduction} 
Let $M$ be a compact, connected, orientable $3$-manifold  with (possibly empty) boundary, and let 
$0\ne \phi \in H^{1}(M; \mathbb{Z}) = \operatorname{Hom}(\pi_{1}(M), \mathbb{Z})$. The \emph{Thurston norm}  is defined as
\[
\|\phi\|_T := \min \{ \chi_{-}(F) \mid F \subset M \text{ properly embedded surface dual to } \phi \},
\]
where, given a surface $F$ with connected components $F_{1}, F_{2}, \ldots, F_{n}$, we define

\[
 \chi_{-}(F) :=  \sum_{i=1}^{n} \max\{-\chi(F_{i}),0\}.
\]
 
Thurston \cite{Thurston1986} proved that this function satisfies the following properties:
\[
\|k\phi\|_T = |k| \|\phi\|_T \quad \text{for } \phi \in H^{1}(M;\mathbb{Z}),\, k \in \mathbb{Z},
\]
\[
\|\phi_1+\phi_2\|_T\le \|\phi_1\|_T+\|\phi_2\|_T\quad \text{for } \phi_{1}, \phi_{2} \in H^{1}(M; \mathbb{Z}).
\]

These properties imply that  $\|\cdot \|_T$  extends to a function on $H^{1}(M;\mathbb{Q})$, which can then be extended by continuity to a seminorm $\|\cdot \|_T$ on $H^{1}(M;\mathbb{R})$.

If $G$ is a finitely generated group and $\phi\in \Hom (G, \QQ)$, there exists a unique $\alpha_\phi\in \QQ_{\ge 0}$ such that $\phi(G)=\alpha_\phi \Z$. A 3-manifold $M$ as above is \emph{admissible} if it is irreducible with (possibly empty) toroidal boundary and its fundamental group is infinite. Friedl and L\"uck~\cite[Theorem 0.2]{FriedlLueck2019a} showed that if $M$ is admissible then for $\phi \in H^{1}(M;\mathbb Q)= H^{1}(\pi_1(M);\mathbb Q)$,
\[
\|\phi\|_T= \alpha_\phi \, b_{1}^{(2)}(\ker \phi).
\]

We say that a group \(G\) is of \emph{type \(\FP_{2}(\ell^{2})\)} if there exists a finitely presented group \(\widetilde{G}\) together with a surjective homomorphism \(\widetilde{G} \twoheadrightarrow G\) inducing an isomorphism
\[
\mathcal{U}(G)\otimes_{\mathbb{Q}\widetilde{G}} I_{\mathbb{Q}\widetilde{G}}
   \;\xrightarrow{\ \cong\ }\;
\mathcal{U}(G)\otimes_{\mathbb{Q} G} I_{\mathbb{Q} G}
\]
where $\UO(G)$ stands for the algebra of unbounded operators affiliated to $G$ (see \cref{sec: l2_betti_nbrs}). Examples of groups of type $\FP_2(\ell^2)$ are groups of type \(\FP_{2}(\mathbb{Q})\) and finitely generated groups satisfying the Weak Atiyah Conjecture  (see \cref{prop:FP2-equivalent}). If \(G\) is of type \(\FP_{2}(\ell^{2})\) and $b_{1}^{(2)}(G)$ vanishes, then according to \cref{propfpl2}(i) for every \(\phi \in H^{1}(G; \mathbb{Q})\),   $b_{1}^{(2)}(\ker \phi)$ is finite. Thus, under these assumptions, we can define a function on $H^{1}(G; \mathbb{Q})$ by
\[
\|\phi\|_{T} := \alpha_\phi \, b_{1}^{(2)}(\ker \phi),
\]
which we will call the \emph{Thurston map}. A natural question is whether the Thurston map satisfies 
\begin{equation}
    \label{convex}
\|\phi_1 + \phi_2\|_T \le \|\phi_1\|_T + \|\phi_2\|_T \quad \text{for } \phi_{1}, \phi_{2} \in H^{1}(G; \mathbb{Q}).
\end{equation}
If the Thurston map is subadditive, it can be extended continuously to a seminorm on $H^{1}(G; \mathbb{R})$, which we call the \emph{Thurston norm}. Our first result provides a large family of groups admitting the Thurston norm.

\begin{theorem}\label{thurston}
Let $G$ be a finitely generated torsion-free group with trivial first $L^{2}$-Betti number and satisfying the Strong Atiyah Conjecture over $\mathbb Q$. Then the Thurston map satisfies \cref{convex}. Moreover, there exists a symmetric integral polytope $P$ in $H_{1}(G; \mathbb{R})$ such that
\[
\|\phi\|_T \;=\;  \max_{x \in P} \phi(x)  
\]
for every character $\phi \colon G \to \mathbb{Q}$.
\end{theorem}

We can generalize the previous result to higher $L^2$-Betti numbers of the kernels of characters. 
This solves a generalized version of a conjecture of Friedl, L\"uck and Tillmann \cite[Conjecture~6.7]{FriedlLueckTillmann2019}.

\begin{theorem}\label{higher}
Let $k$ be a positive integer and let $G$ be a torsion-free group of type $\mathrm{FP}_k(\QQ)$ with trivial $k$-th $L^{2}$-Betti number and satisfying the Strong Atiyah Conjecture over $\mathbb Q$. Then there exists a symmetric integral  polytope $P$ in $H_{1}(G; \mathbb{R})$ such that
\[
b_k^{(2)}(\ker \phi) \;=\;  \max_{x \in P} \phi(x)  
\]
for every epimorphism $\phi \colon G \to \mathbb{Z}$.
\end{theorem}

When the group $G$ is of type $\mathrm{FP}_2(\QQ)$ (e.g., finitely presented), then by \cite[Theorem A]{BieriStrebel1978}, every $0 \neq \phi \in H^1(G; \Z)$ induces a group splitting of $G$, namely $G$ decomposes as an HNN extension $G = K \ast_{H}$ over finitely generated subgroups $K$ and $H$ such that $K \le \ker \phi$. We call such a decomposition an \emph{admissible splitting dual to $\phi$}. By \cref{propfpl2}(ii), if $b_1^{(2)}(G) = 0$, we have that  
\[
    b_1^{(2)}(\ker \phi) \leq  b_1^{(2)}(K)\leq b_1^{(2)}(H).
\]
The number $b_1^{(2)}(H)$ is the \emph{complexity} of the splitting $G = K \ast_{H}$.
It was suggested (see, for example, \cite[Section~5.1]{FriedlLueckTillmann2019}) that in order to find a combinatorial interpretation of the Thurston norm, one should look at the \emph{splitting complexity} of the character:
\begin{equation}\label{eq:splitting-complexity}
c(G; \phi) := \inf \left\{ b_1^{(2)}(H) \,\middle|\, G = K \ast_{H} \text{ is an admissible splitting dual to } \phi \right\}.
\end{equation}
Observe that if the group $G$ satisfies the Weak Atiyah Conjecture, then the infimum in the previous definition is in fact a minimum.
This leads us to the following problem:

\begin{problem}\label{problem}
Let $G$ be a  group of type $\mathrm{FP}_2(\QQ)$ with trivial first  $L^2$-Betti number (satisfying the Weak Atiyah Conjecture). Is it true that  
\[
b_1^{(2)}(\ker \phi) =   c(G; \phi)
\]
for every non-trivial $\phi : G \to \mathbb{Z}$?
\end{problem}

We record that for a free-by-cyclic group $G$, this question was asked explicitly in \cite{Oberwolfach} by Gardam and Kielak.
 
 Now we explain our approach to answer the question positively.
Let $K$ be a finitely generated subgroup of $G$ and $H$ a finitely generated subgroup of $K$. We say that $H$ is \emph{$L^2$-independent} in $K$ if the natural map
\[
\mathcal U(K) \otimes_{\mathbb{Q} H} I_{\mathbb{Q} H} \longrightarrow \mathcal U(K) \otimes_{\mathbb{Q} K} I_{\mathbb{Q} K}
\]
is $\dim_{\mathcal U(K)}$-injective (see \cref{subsec:slyv}). We say that  $H$ is \emph{$L^2$-compressed} in $K$ if for any  finitely generated subgroup $L$ of $K$ containing $H$ it holds that
\[
b_1^{(2)}(H) \leq b_1^{(2)}(L).
\]
Clearly, if $H$ is $L^2$-independent in $K$, then $H$ is also $L^2$-compressed in $K$. The following theorem summarizes our approach towards a positive answer to the question in \cref{problem}.

\begin{theorem}\label{approach}
Let $G$ be a torsion-free group of type $\mathrm{FP}_2(\QQ)$ with trivial first  $L^2$-Betti number. 
Let $0 \ne \phi \in H^{1}(G; \mathbb{Z})$.
\begin{enumerate}
    \item \label{item: approach_b1_equality}
  Let  $G = K \ast_{H,\theta}$ be  an {admissible splitting dual to $\phi$}. If $H$ and $\theta(H)$ are $L^2$-independent in $K$, then 
    \[
    b_1^{(2)}(\ker \phi) = b_1^{(2)}(H).
    \]
    \item \label{item: approach_l2indep_splitting}
    Assume $G$ satisfies the Weak Atiyah Conjecture and that for every pair of finitely generated subgroups $H \leq K \leq \ker \phi$, if $H$ is $L^2$-compressed in $K$, then it is also $L^2$-independent. Then $G$ admits an \emph{admissible splitting dual to $\phi$}, namely $G = K \ast_{H,\theta}$, such that both $H$ and $\theta(H)$ are $L^2$-independent in $K$.
\end{enumerate}
\end{theorem}

In other words, \cref{problem} is reduced via \cref{approach} to the following problem.

\begin{problem}\label{problem2}
    Let $ K$ be a  finitely generated group and $H$ a subgroup of $K$. Assume that $H$ is $L^2$-compressed in $K$. Is it true that $H$ is $L^2$-independent in $K$?
\end{problem}

In \cite{Ja24}, the first author answered the question in \cref {problem2} in the affirmative for free groups. Here we solve \cref{problem2} for torsion-free virtually free-by-cyclic groups.

\begin{theorem}\label{compressedL2}
Let $K$ be a torsion-free virtually free-by-cyclic group.
Let $H$ be a finitely generated subgroup of $K$.  
If $H$ is $L^2$-compressed in $K$, then $H$ is also $L^2$-independent in $K$.
\end{theorem}

By combining \cref{approach,compressedL2}, we obtain the following corollary.

\begin{corollary}\label{goodsplittings}
Let $G$ be either
\begin{enumerate}
    \item a torsion-free virtually (finitely generated free)-by-cyclic group; or
    \item the fundamental group $\pi_1(M)$ of an admissible $3$-manifold $M$.   
\end{enumerate}
Then for every non-trivial homomorphism $\phi \colon G \to \mathbb{Z}$, the group $G$ admits an admissible splitting dual to $\phi$, namely
\[
G = K \ast_H,
\]
such that
\[
b_1^{(2)}(\ker \phi) = b_1^{(2)}(H).
\]
\end{corollary}
We note that the case of $3$-manifold groups is also a consequence of the result of Thurston \cite{Thurston1986} mentioned at the beginning of the introduction and its interpretation by Friedl and L\"uck \cite{FriedlLueck2019a}.

In the proof of \cref{compressedL2}, we also establish the following result, which is of independent interest.

\begin{theorem}\label{thm: univ_vFbyZ}
Let $E$ be a division ring, let $G$ be a torsion-free virtually free-by-cyclic group, and let $E * G$ be a crossed product. Then $E * G$ is a pseudo-Sylvester domain. In particular, $G$ is a Lewin group.
\end{theorem}

This generalizes the $3$-manifold group case proved in \cite{SanchezPeralta_Knots}, complementing earlier work that addressed the free-by-cyclic special case \cite{HL22} and the pro-$p$ analogue \cite{JS26}.

We end the article with several applications of \cref{goodsplittings}. Recall that for a finitely generated group $G$, the Bieri--Neumann--Strebel (BNS) invariant is the set of non-trivial characters $\phi \colon G \to \mathbb{R}$ such that the monoid $\{g\in G \mid \phi(g) \geq 0\}$ is of type $\FP_1$. The BNS invariant detects finite generation of kernels of integral characters \cite[Proposition~4.4]{BieriNeumannStrebel1987}. 

If $M$ is an admissible 3-manifold, then the BNS invariant of $\pi_1(M)$ can be computed algorithmically as in \cite{CooperTillmann2009, Schleimer2001}. We outline the procedure.

By the work of Bieri, Neumann and Strebel \cite[Theorem~E]{BieriNeumannStrebel1987}, $\Sigma(\pi_1(M))$ is precisely the union of cones over the interiors of certain top-dimensional faces of the Thurston norm ball $B_T(M)$. Consequently, to compute $\Sigma(\pi_1(M))$, it suffices to compute $B_T(M)$ and check a finite number of integral characters $\phi \colon \pi_1(M)\to \mathbb{Z}$, one for each top-dimensional face. Moreover, an integral character $\phi \colon \pi_1(M) \to \mathbb{Z}$ is an element of $\Sigma(\pi_1(M))$ if and only if it is represented by a fibration $M \to \mathbb{S}^1$. To check whether this is the case, one constructs an embedded surface $F \subseteq M$ dual to $\phi$ and checks whether the complement $\overline{M \setminus F}$ is homeomorphic to $F \times [0,1]$ via Haken's algorithm. Finally, an algorithm to compute $B_T(M)$ is given by Tollefson and Wang in \cite[Algorithm~5.9]{TollefsonWang1996}.

If $G$ is a (finitely generated free)-by-cyclic group, it was shown by Kielak that the BNS invariant of $G$ is also determined by an integral polytope $P \subseteq H_1(G; \mathbb{R})$ \cite{Kielak2020polytopes}. More precisely, there exists a marking of vertices of $P$ such that a character $\phi \in H^1(G;\mathbb{R})$ is in the BNS invariant if and only if $\phi|_P$ attains its minimum precisely at a marked vertex. 

Consequently, in \cite{Oberwolfach} Gardam--Kielak sketched an algorithm which computes the BNS invariant of \fbyz groups, conditional on a positive answer to \cref{problem} for such groups. We fill in the details of their sketch and combine it with \cref{goodsplittings} to obtain the following: 

\begin{corollary}[Gardam--Kielak]\label{GardamKielak}
Let $G$ be a \fbyz group. Then, there exists an algorithm to compute the BNS invariant of $G$.
\end{corollary}

The second application concerns the isomorphism problem in (finitely generated free)-by-cyclic groups. Outside of the hyperbolic and toral-relatively-hyperbolic setting, little is known about the isomorphism problem for this family.

Let $\mathbb{F}_n$ denote the free group of rank $n$. We apply the methods used in the proof of \cref{GardamKielak} to show that there exists an algorithm to compute the set of epimorphisms $\phi \colon G \to \mathbb{Z}$ with $\ker \phi \cong \mathbb{F}_n$ for any \fbyz group $G$ and $n \in \mathbb{N}$. In the case that $G$ does not split as an HNN extension over a \fbyz subgroup, this allows us to reduce the isomorphism problem to the conjugacy problem in $\mathrm{Out}(\mathbb{F}_n)$ (\cref{cor:conjugacy-implies-iso}). We also discuss the general case. While a complete solution to the conjugacy problem in $\Out(\mathbb{F}_n)$ is currently unknown, much progress has been made in recent years towards partial solutions \cite{CohenLustig1995,Los1996, KLM2001, FeighnHandel2025, DFMT2025, Bartlett2026}.

\smallskip

The paper is organized as follows. In \cref{sec:prelims} we outline the necessary preliminaries. In \cref{FP2L2}, we establish bounds on the first $L^2$-Betti number of kernels of integral characters in terms of the first $L^2$-Betti numbers of the subgroups in an admissible dual splitting, and prove a generalisation of \cref{approach} (\cref{propfpl2}). In \cref{sec:thurston} we prove \cref{higher} from which \cref{thurston} follows immediately. The aim of \cref{sec:good} is to establish cohomological goodness for finitely generated free-by-cyclic groups (\cref{prop: goodness_FbyZ}). In \cref{sec:Lewin} we study the properties of crossed products of division rings with torsion-free virtually free-by-cyclic groups, in particular establishing \cref{thm: univ_vFbyZ}. In \cref{sec: l2_rigid} we show that every torsion-free virtually free-by-cyclic group is $L^2$-subgroup rigid (\cref{thm:L2rigid}) and deduce \cref{compressedL2}. We end this section by proving \cref{goodsplittings}. Finally in \cref{sec:applications} we establish \cref{GardamKielak} on computing the BNS invariant for \fbyz groups. We also discuss applications to the isomorphism problem for \fbyz groups.

\subsection*{Acknowledgments} 
We thank Giles Gardam and Dawid Kielak for allowing us to reproduce their sketch of \cref{GardamKielak} and for helpful discussions.
The second author would like to thank Misha Schmalian for helpful discussions about 3-manifolds. The work of the first and third authors is partially supported by the grants  \seqsplit{PID2024-155800NB-C33} and \seqsplit{EUR2025-164928} of the Ministry of Science, Innovation and Universities of Spain.

\section{Preliminaries}\label{sec:prelims}
\subsection{Notation}

    Throughout, groups are countable, rings are assumed to be associative and unital, ring homomorphisms preserve the unit and modules are left modules unless otherwise specified.

    If $H$ is a subgroup of $G$ and $M$ is an $RH$-module, we denote by $^G M$ the induced $RG$-module $RG \otimes_{RH} M$.

    We use $R^{\times}$ to denote the multiplicative group of units in a ring $R$.

\subsection{Augmentation ideals}

    Given a group $G$ and a ring $R$, we denote the augmentation ideal of $RG$ by $I_{RG}$. If $H$ is a subgroup of $G$, we note that $^G I_{RH}$ agrees with the $RG$-submodule of $RG$ generated by $I_{RH}$. We set $I_R(G,H)$ to be the $RG$-module $I_{RG}/^G I_{RH}$.
    
    Due to its independent interest, we record two properties of augmentation ideals that are central to our work. The first property examines the behaviour of the augmentation ideal with respect to descending chains of subgroups.
    
    \begin{lemma} \label{lem: inter_aug_ideals}
        Let $G$ be a group and $\{L_i\}_{i \in I}$ a linearly ordered descending chain of subgroups of $G$. Set $L = \cap_{i \in I} L_i$. Then
        \[
            \bigcap_{i \in I} {}^G I_{RL_i} = {}^G I_{RL}.
        \]
    \end{lemma}
    \begin{proof}
        The right inclusion is clear, so suppose for a contradiction that $a \in \cap_{i \in I} {}^G I_{RL_i}$ is such that $a \notin {}^{G} I_{RL}$. Then we can write
        \[
        a = \sum_{j=1}^n g_j a_j,
        \]
        where $0 \ne a_j \in RL$, and the elements $g_j$ lie in distinct left $L$-cosets with at least one coefficient, say $a_1$, not belonging to $I_{RL}$. Since $L = \bigcap_{i \in I} L_i$, there exists $i_0 \in I$ such that the elements $\{ g_j : 1 \leq j \leq n \}$ lie in distinct left $L_{i_0}$-cosets. Hence, $a \notin {}^{G} I_{RL_{i_0}}$, a contradiction.
    \end{proof}

    The second property shows the advantage of working with coefficients in $\FF_2$ in \cref{subsec: F2_rigid}.

    \begin{lemma}(\cite[Lemma 3.6]{Ja24}) \label{lem: index2_aug_ideal}
        Let $G$ be a finitely generated group and let $M$ be an $\FF_2 G$-submodule of $I_{\FF_2 G}$ such that $I_{\FF_2 G} / M$ is isomorphic to $\FF_2$ with trivial $G$-action. Then there exists an index $2$ subgroup $U$ of $G$ satisfying $M = {}^G I_{\FF_2 U}$.
    \end{lemma}
    \begin{proof}
        Consider the group homomorphism $\varphi \colon G \to \FF_2$ that sends $g$ to the class of $g - 1$ in $I_{\FF_2 G} / M$. Then $U = \ker \varphi$ since over $\FF_2$ the action of $G$ on $I_{\FF_2}(G, U)$ is trivial.
    \end{proof}

\subsection{Von Neumann regular and \texorpdfstring{$\ast$}--regular rings}

    A ring $\UO$ is called {\it von Neumann regular} if for every $x\in\mathcal U$ there exists $y\in\mathcal U$ such that $xyx =x$. From a homological point of view, these rings possess a nice feature:

    \begin{proposition}[{\cite[Theorem 1.11 \& Proposition 2.6]{GoodearlvNr}}] \label{prop: vNr_proj_mod}
        Let $\UO$ be a von Neumann regular ring and $M$ a finitely presented $\UO$-module. Then $M$ is projective and admits a direct sum decomposition into cyclic projective $\UO$-modules.
    \end{proposition}

    The von Neumann regular rings we shall be working with have an additional structure, namely they are $\ast$-regular rings. A {\it $\ast$-ring} is a ring $R$ with a map $\ast: R \rightarrow R$ that is an involution, i.e. $(x^{\ast})^{\ast} = x, (x+y)^{\ast} = x^{\ast} + y^{\ast}$ and $(xy)^{\ast} = y^{\ast} x^{\ast}$ for $x,y\in R$. The involution is called {\it proper} if $x^{\ast} x = 0$ implies $x = 0$. A {\it $\ast$-regular ring} is a von Neumann regular ring with a proper involution.

    If $\UO$ is a $\ast$-regular ring, then for each $x \in\mathcal U$ there is a unique element $y \in\mathcal U$, called the \textit{relative inverse of} $x$, such that $xyx = x$, $yxy = y$ and both $xy$ and $yx$ are projections, that is idempotents and self-adjoints; see \cite[Proposition 51.4(i)]{Berberian72}. For a $\ast$-subring $R$ of a $\ast$-regular ring $\mathcal{U}$, P. Ara and K. R. Goodearl realized that one can construct the smallest $\ast$-regular subring of $\mathcal{U}$ containing $R$, termed the \textit{$\ast$-regular closure} and denoted by $\mathcal{R}(R, \mathcal{U})$, in the same spirit of the division closure of a subring of a ring. The construction in \cite[Proposition 6.2]{AraGoodearlRegClosure} goes as follows:
    \begin{enumerate}[label=(\roman*)]
        \item Set $\mathcal{R}_0(R, \mathcal{U}) := R$, a $\ast$-subring of $\UO$.
        \item Suppose that $n \geq 0$ and that we have constructed the $\ast$-subring $\mathcal{R}_n(R,\mathcal U)$ of $\UO$. Let $\mathcal{R}_{n+1}(R,\mathcal U)$ be the subring of $\UO$ generated by the elements of $\mathcal{R}_n(R,\mathcal U)$ and their relative inverses in $\UO$, which is a $\ast$-subring.
    \end{enumerate}
    Then $\mathcal{R}(R,\mathcal U) = \cup_{n = 0}^{\infty} \mathcal{R}_n(R,\mathcal U)$.

\subsection{Sylvester rank functions}\label{subsec:slyv}

    Let $R$ be a ring. A \textit{Sylvester matrix rank function} $\rk$ on $R$ is a function that assigns a non-negative real number to each matrix over $R$ and satisfies the following conditions:
    \begin{enumerate}[label=(\roman*)]
        \item $\rk(A)=0$ if $A$ is any zero matrix and $\rk(1)=1$;
        \item $\rk(AB)\leq \min\{\rk(A),\rk(B)\}$ for any matrices $A$ and $B$ which can be multiplied;
        \item $\rk\begin{psmallmatrix}
             A & 0 \\
             0 & B
        \end{psmallmatrix}=\rk(A)+\rk(B)$ for any matrices $A$ and $B$; and
        \item $\rk \begin{psmallmatrix}
             A & C \\
             0 & B
        \end{psmallmatrix} \geq \rk(A)+\rk(B)$ for any matrices $A$, $B$ and $C$ of appropriate sizes.
    \end{enumerate}
    
    We denote by $\mathbb{P}(R)$ the set of Sylvester matrix rank functions on $R$. Note that a ring homomorphism $\varphi\colon R\rightarrow S$ induces a map $\varphi^{\#}\colon \mathbb{P}(S)\rightarrow \mathbb{P}(R)$. Indeed, we can pull back any rank function $\rk$ on $S$ to a rank function $\varphi^{\#}(\rk)$ on $R$ by setting
    \[
    \varphi^{\#}(\rk)(A):=\rk(\varphi(A))
    \]
    for every matrix $A$ over $R$. We shall often abuse notation and write $\rk$ instead of $\varphi^{\#}(\rk)$ when it is clear that we are referring to the rank function on $R$. There is a dual notion of Sylvester module rank function. 
    
    \begin{definition}\label{def: SMRF}
        A {\it Sylvester module rank function}, or simply {\it $\SMRF$}, $\dim$ on $R$ is a map that assigns a non-negative real number to each finitely presented $R$-module and satisfies the following properties:
        \begin{enumerate}[label=(\roman*)]
            \item $\dim(0)=0$ and $\dim(R)=1$;
            \item $\dim(M_1\oplus M_2)=\dim(M_1)+\dim(M_2)$; and
            \item if $M_1\rightarrow M_2\rightarrow M_3\rightarrow 0$ is exact, then
            \[
            \dim(M_1)+\dim(M_3)\geq \dim(M_2)\geq \dim(M_3). 
            \]
        \end{enumerate}
    \end{definition}

    As one might expect, there is a bijective correspondence between these two notions defined as follows. Given a Sylvester matrix rank function $\rk$ on $R$, we can define an $\SMRF$ $\dim$ on $R$ by assigning to any finitely presented $R$-module with presentation $M=R^m/R^nA$ for some $A\in \Mat_{n\times m}(R)$, the value $\dim(M):=m-\rk(A)$; conversely, if $\dim$ is a $\SMRF$ on $R$, then we obtain a Sylvester matrix rank function setting $\rk(A):=m-\dim(R^m/R^nA)$ (cf. \cite[Theorem 4]{Malcolmson} and \cite[Proposition 1.2.8]{DLopezAlvarezThesis}).

    Given an $\SMRF$ $\dim$ on $R$, for a finitely generated $R$-module $M$ one can define
    \[
        \dim(M) := \inf \{ \dim(M') : M' \twoheadrightarrow M \mbox{ and } M' \mbox{ is a finitely presented $R$-module} \}.
    \]
    The definition is well defined and extends $\dim$ for finitely presented $R$-modules (see \cite[Lemma 3.8]{Li_SMRF}). The $\SMRF$ $\dim$ is \emph{exact} if for any exact sequence of $R$-modules 
    \[
        0 \to M_1 \to M_2 \to M_3 \to 0
    \]
    such that $M_2$ and $M_3$ are finitely presented (and hence, $M_1$ is finitely generated), it holds that $\dim(M_2) = \dim(M_1) + \dim(M_3)$. Note that by \cref{prop: vNr_proj_mod} this is always the case for von Neumann regular rings. 

    The definition of exact $\SMRF$ was introduced by S. Virili in \cite{Virili_LengthFunc}, where he proved that an exact $\SMRF$ extends uniquely to a normalized length function, and thus, for an arbitrary $R$-module $M$
    \[
        \dim(M) := \sup \{ \dim(M') : M' \mbox{ is a finitely generated $R$-submodule of $M$} \}
    \]
    is well defined. In general, Li showed in \cite{Li_SMRF} that any $\SMRF$ can be extended to a function that computes the dimension of arbitrary modules via bivariant Sylvester module rank functions extending the work of Virili on exact rank functions. Thus, from now on we shall abuse notation and denote by $\dim$ the extension to arbitrary modules of a $\SMRF$ $\dim$. For a general $\SMRF$ only the following holds (see \cite[Lemma 3.1]{SP_L2}).

    \begin{lemma} \label{lem: dim_direct_lim}
        Let $R$ be a ring with $\SMRF$ $\dim_R$ and $M$ an $R$-module. If $M = \varinjlim_{i\in I} M_i$ is a direct limit of $R$-modules, then
        \[
            \dim_R(M)\leq \sup_{i\in I} \{ \dim_R(M_i) \}.
        \]
    \end{lemma}

    Let $R$ be a ring with $\SMRF$ $\dim_R$. An $R$-homomorphism $f \colon M \rightarrow N$ between $R$-modules $M,N$ is {\it $\dim_R$-injective}, respectively {\it $\dim_R$-surjective}, if $\dim_R(\ker f)$, respectively $\dim_R(\coker f)$, vanishes. If $f$ is both $\dim_R$-injective and $\dim_R$-surjective, we say that $f$ is a {\it $\dim_R$-isomorphism} and that the modules $M$ and $N$ are {\it $\dim_R$-isomorphic}. We record a convenient result.

    \begin{lemma}[{\cite[Lemma 3.5]{SP_L2}}]\label{lem: comp_dim_isomph}
        Let $R$ be a ring with exact $\SMRF$ $\dim_R$ and consider two $\dim_R$-isomorphisms $f \colon M_1\rightarrow M_2$ and $g \colon M_2\rightarrow M_3$. Then $g \circ f \colon M_1\rightarrow M_3$ is a $\dim_R$-isomorphism.
    \end{lemma}

\subsection{\texorpdfstring{$L^2$}--Betti numbers} \label{sec: l2_betti_nbrs}

    Let $G$ be a countable group and let $\ell^2(G)$ denote the Hilbert space with basis the elements of $G$, that is, $\ell^2(G)$ consists of all square-summable formal sums
    \[
        \sum_{g \in G} a_g g
    \]
    with $a_g \in \mathbb{C}$, and inner product
    \[
        \left\langle \sum_{g \in G} a_g g, \sum_{g \in G} b_g g \right\rangle = \sum_{g \in G} a_g \overline{b_g}.
    \]
    The left and right multiplication actions of $G$ on itself extend to left and right actions of $G$ on $\ell^2(G)$. We can further extend the right action of $G$ on $\ell^2(G)$ to an action of $\mathbb{C}G$ on $\ell^2(G)$, and hence consider $\CC G$ as a subalgebra of $\mathcal{B}(\ell^2(G))$, namely the {\it bounded linear operators on} $\ell^2(G)$.

    The {\it group von Neumann algebra} $\VN(G)$ of $G$ is the algebra of left $G$-equivariant bounded operators on $\ell^2(G)$
    \[
        \VN(G):=\{\phi\in \mathcal{B}(\ell^2(G)): \phi(g\cdot v) = g \cdot \phi(v) \mbox{ for all }g\in G,v\in \ell^2(G)\}.
    \]
    The ring $\VN(G)$ is a finite von Neumann algebra. Moreover, $\VN(G)$ satisfies the left (and right) Ore condition with respect to the set of non-zero-divisors (a result proved by S. K. Berberian in \cite{Berberian82}). We denote by $\UO(G)$ the left (resp. right) classical ring of fractions named the {\it ring of unbounded operators affiliated to} $G$. The reader can consult the Ph.D thesis of H. Reich \cite{ReichThesis} for the following standard facts. The ring $\UO(G)$ can be defined alternatively as the set of densely defined (unbounded) operators which are closed, i.e. its graph is closed in $\ell^2(G)^2$, and commute with the left action of $G$. It is a $\ast$-regular ring where the involution is given by taking the adjoint operator. Note that we have the following chain of $\ast$-embeddings
    \[
        \CC G \subseteq \VN(G) \subseteq\mathcal U(G).
    \]

    A finitely generated {\it Hilbert} $G$-module is a closed subspace $V\leq \ell^2(G)^n$ invariant under the left action of $G$. Given such $V$, there exists $\pr_V$, the orthogonal projection onto $V$. We set 
    \[
    \dim_G(V):=\trace_G(\pr_V):=\sum_{i=1}^n \langle \pr_V(1_i),1_i\rangle,
    \]
    where $1_i$ is the element of $\ell^2(G)^n$ having $1$ in the $i$th entry and $0$ in the rest of the entries. We call this number the {\it von Neumann dimension} of $V$. This dimension is independent of the $G$-isometric embedding and induces a Sylvester matrix rank function on $\CC G$. Concretely, given $A\in \Mat_{n\times m}(\CC G)$ we can then associate the natural bounded linear operator $\phi_G^A:\ell^2(G)^n\rightarrow\ell^2(G)^m$ given by right multiplication, and define $\rk_G(A)=\dim_G \overline{\im \phi_G^A}$ where the bar stands for the closure of the image subspace. The Sylvester matrix rank function $\rk_G$ can be extended to $\UO(G)$ as follows
    \[
        \rk_G(s^{-1}r):=\rk_G(r)=\dim_G \left( \overline{\ell^2(G) r} \right).
    \]
    This Sylvester matrix rank function is indeed the one associated to the $L^2$-dimension $\dim_{\scalebox{0.8}{$\UO(G)$}}$ \cite[Definition 8.28]{Lueck2002} (cf. \cite[Lemma 4.2.4]{DLopezAlvarezThesis}). The computation of $L^2$-Betti numbers has been algebraized through the seminal works of L\"uck \cite{Lu88I, Lu88II} and the thesis of Reich \cite{ReichThesis}. The basic observation is that since $\UO(G)$ is von Neumann regular, one can define as in \cref{subsec:slyv} the $L^2$-dimension over all modules of $\mathcal{U}(G)$. For the convenience of the reader we record two properties that we will use throughout the text.

    \begin{proposition}\label{prop: properties_L2_dim}
        Let $G$ be a group. Then the following holds:
        \begin{enumerate}
            \item For any subgroup $H$ in $G$ and any $\UO(H)$-module $M$, it holds that
            \[
                \dim_{\UO(H)}(M) = \dim_{\UO(G)} \left(\UO(G)\otimes_{\UO(H)}M \right).
            \]
            \item For any finitely presented $\UO(G)$-module $M$, if $\dim_{\UO(G)}(M)$ vanishes, then $M = \{0\}$.
        \end{enumerate}
    \end{proposition}
    
    Given a non-negative integer $i$ and a $\mathbb{C}G$-module $M$, we define the \emph{$i$th $L^2$-Betti number of $M$} using the following formula:
    \[
        \beta^{\mathbb{C}G}_i(M) := \dim_{\mathcal{U}(G)} \left( \operatorname{Tor}^{\mathbb{C}G}_i(\mathcal{U}(G), M) \right).
    \]

    There is another chain of $\ast$-embeddings produced by the $\ast$-regular closure. Specifically, since $\UO(G)$ is a $\ast$-regular ring, given $k \subseteq \CC$ a field closed under complex conjugation, we set $\mathcal{R}_{k G}$ to be the $\ast$-regular closure of $kG$ in $\mathcal{U}(G)$, or in other words, $\mathcal{R}_{kG}:= \mathcal{R}(k G,\mathcal U(G))$. Then 
    \[
        k G \subseteq \mathcal{R}_{k G} \subseteq\mathcal U(G)
    \]
    is another such chain. We will denote $\mathcal{R}_{\CC G}$ by $\mathcal{R}_G$. Similarly, if $k$ is a subfield of $\CC$ and $M$ is a $k G$-module, then its $L^2$-Betti numbers are computed using the formula
    \[
        \beta^{k G}_i(M) = \dim_{\mathcal{R}_G} \left( \operatorname{Tor}^{k G}_i(\mathcal{R}_G, M) \right),
    \]
    where $\dim_{\mathcal{R}_G}$ is induced from the $L^2$-dimension.
    
    The {\it Strong Atiyah Conjecture} (over $k$) predicts that if $\operatorname{lcm}(G)$, the least common multiple of the orders of finite subgroups of $G$, is finite, then for every finitely presented $k G$-module $M$
    \[
        \beta^{{k}G}_0(M) \in \frac{1}{\operatorname{lcm}(G)} \mathbb{Z}.
    \]
    Given a subfield $k$ of $\CC$, \emph{the Linnell ring over $k$}, denoted $\mathcal{D}(k[G])$, is the division closure of $k[G]$ in $\UO(G)$. We recall that when $\operatorname{lcm}(G)=1$, the conjecture is equivalent to the Linnell ring being a division ring. In particular, if $k$ is closed under complex conjugation, then $\mathcal{R}_{kG}=\mathcal{D}(kG)$ is a division ring.

    In this paper we will also use a weaker version of the Atiyah conjecture. A group $G$ satisfies the {\it Weak Atiyah Conjecture} (over $k$) if there exists $l \in \mathbb{N}$ such that for every finitely presented $kG$-module $M$
    \begin{equation}
        \label{wAtiyah}
        \beta^{{k}G}_0(M) \in \frac{1}{l} \mathbb{Z}.
        \end{equation}
    Note that if $k$ is closed under complex conjugation, this is equivalent to the ring $\mathcal{R}_{kG}$ being a semisimple (and hence Artinian) algebra. Indeed, by \cite[Corollary~6.2]{Ja19}, the condition (\ref{wAtiyah}) implies that 
    \[
        \rk_{G}(a)\in \frac{1}{l}\mathbb{Z} \quad \text{for every } a\in \mathcal{R}_{kG}.
    \]
    Since $\mathcal{R}_{k G}$ is von Neumann regular and $\rk_{G}$ is faithful, \cref{prop: vNr_proj_mod} implies that in this case  $\mathcal{R}_{k G}$ is semisimple.

    There exist groups for which the Weak Atiyah Conjecture is known to hold, while the strong version remains open like linear groups over a field of characteristic $0$ or virtually locally indicable groups. This distinction arises because a group that virtually satisfies the Weak Atiyah Conjecture automatically satisfies it itself, whereas this inheritance property is not known for the Strong Atiyah Conjecture.

\subsection{Division ring embeddings} \label{sec: div_ring_embed}

    An \textit{$R$-ring} is a pair $(S, \varphi)$ where $\varphi \colon R \rightarrow S$ is a ring homomorphism. We will often omit $\varphi$ if it is clear from the context. Two $R$-rings $(S_1, \varphi_1)$ and $(S_2, \varphi_2)$ are {\it isomorphic} if there exists a ring isomorphism $\alpha \colon S_1 \to S_2$ such that $\varphi_2 = \alpha \circ \varphi_1$. 

    A \emph{division} $R$-ring $(\mathcal{D}, \varphi)$ is an $R$-ring such that $\mathcal{D}$ is a division ring. Note that $\mathcal D$ inherits the structure of an $R$-bimodule, and therefore given an $R$-module $M$ we can compute its homology with coefficients in $\mathcal D$. Recall that every module over $\mathcal D$ has a well-defined dimension.

    \begin{definition}
        Let $\mathcal D$ be a division $R$-ring, $M$ an $R$-module and $n$ a non-negative integer. The $n$th \emph{$\mathcal D$-Betti number of $M$} is the (possibly infinite) value
        \[
            \beta_n^\mathcal D (M) := \dim_\mathcal D \left( \Tor_n^R (\mathcal D, M) \right) \in \Z_{\geq 0} \cup \{\infty\}.
        \]
    \end{definition}

    Observe that when $G$ is a torsion-free group satisfying the Strong Atiyah Conjecture, the Linnell ring $\mathcal{D}(\CC G)$ is a division $\CC[G]$-ring and for every non-negative integer $n$
    \[
        \beta_n^{\mathcal{D}(\CC G)}(\CC) = b_n^{(2)}(G)
    \]
    where $G$ acts trivially on $\CC$. In this case, we shall denote by $\beta_n^G$ the $\mathcal{D}(\QQ[G])$-Betti numbers of $\QQ[G]$-modules.

    An \emph{epic} division $R$-ring is a division $R$-ring $(\mathcal{D}, \varphi)$ such that the division closure of $\im \varphi$ equals $\mathcal{D}$. It is further called \emph{division $R$-ring of fractions} if $\varphi$ is injective.

    Let $G$ be a group. A ring $S$ is $G$-\emph{graded} if $S = \bigoplus_{g\in G} S_g$ as an additive group, where $S_g$ is an additive subgroup for every $g\in G$, and $S_g S_h \subseteq S_{gh}$ for all $g,h \in G$. If $S_g$ contains an invertible element $u_g$ for each $g\in G$, then we say that $S$ is a \emph{crossed product} of $S_e$ and $G$ and we shall denote it by $S = S_e * G$.

    Let $E$ be a division ring, $G$ a locally indicable group and $E * G$ a crossed product. An epic division $E * G$-ring $(\mathcal{D}, \varphi)$ is called \emph{Hughes-free} if for every non-trivial finitely generated subgroup $H \leqslant G$ and every $U \triangleleft H$ such that $H/U \cong \Z$, the multiplication map 
    \[
        \mathcal D_U \otimes_{E * U} E * H \rightarrow \mathcal D, \qquad x \otimes y \mapsto x \cdot \varphi(y)
    \]
    is injective. Here, $\mathcal D_U$ denotes the division closure of $\varphi(E * U)$ in $\mathcal D$. In this case, we call $\mathcal D$ a \emph{Hughes-free division ring} for $E * G$. Hughes \cite{HughesDivRings1970} proved that if $E * G$ admits a Hughes-free embedding, then it is unique up to $E * G$-ring isomorphism. In this situation, we thus speak of \emph{the} Hughes-free division ring of $E * G$, denote it by $\mathcal{D}_{E * G}$, and view $E * G$ as a subring of $\mathcal{D}_{E * G}$. If $H$ is any subgroup of $G$, then $\mathcal{D}_{E * H}$ is isomorphic to the division closure of $E * H$ in $\mathcal{D}_{E * G}$, so we view $\mathcal{D}_{E * H}$ as a subring of $\mathcal{D}_{E * G}$. If the Hughes-free division ring exists for every for every division ring $E$ and every crossed product $E * G$, we say that $G$ is \emph{Hughes-free embeddable}.

    When $G$ is locally indicable, $\mathcal{D}_{\QQ G}$ always exists and coincides with $\mathcal{D}(\QQ G)$ by a result of Jaikin-Zapirain and L\'opez-\'Alvarez \cite{JaikinLopez2020}. Note that by uniqueness up to $\QQ G$-isomorphism, if $G$ is Hughes-free embeddable then $\mathcal{D}(\QQ G)$ is in fact the Hughes-free division $\QQ G$-ring of fractions.
    
    For a general countable group $G$, an epic division $E * G$-ring $(\mathcal{D}, \varphi)$ is called \emph{Linnell-free} if for every pair of subgroups $U \leqslant H$ of $G$, the multiplication map 
    \[
        \mathcal D_U \otimes_{E * U} E * H \rightarrow \mathcal D, \qquad x \otimes y \mapsto x \cdot \varphi(y)
    \]
    is injective. It is shown in \cite{Linnell_localizations} (cf. the discussion after Problem 4.5) that elements in different cosets of $H$ in $G$ are right $\UO(H)$-linearly independent in $\UO(G)$. Thus, if $G$ satisfies the Strong Atiyah Conjecture, $\mathcal{D}_G$ is Linnell-free. However, while \cite{HughesDivRings1970} ensures the uniqueness of the division ring for locally indicable groups, it remains unknown whether Linnell-free division $E*G$-rings are unique up to $E*G$-isomorphism for a general group $G$. Moreover, Gr\"ater showed in \cite[Corollary 8.3]{Grater20} that if $G$ is a locally indicable group and the Hughes-free division $E* G$-ring $\mathcal D_{E*G}$ exists, then it is in fact a Linnell-free division $E*G$-ring.

    A division $R$-ring $\mathcal{D}$ is \emph{universal} if for every division $R$-ring $\mathcal{E}$ and every finitely presented $R$-module $M$ it holds that 
    \[
        \dim_{\mathcal{D}} \left( \mathcal{D} \otimes_{R} M \right) \leq \dim_{\mathcal{E}} \left( \mathcal{E} \otimes_{R} M \right).
    \]
    If it exists, it is unique up to $R$-isomorphism \cite[Theorem 7.2.7]{Cohn2006}. A locally indicable group $G$ is a \emph{Lewin group} if it is Hughes-free embeddable and for all possible crossed products $E * G$, where $E$ is a division ring, $\mathcal{D}_{E * G}$ is universal. Prominent examples of rings having universal division rings are pseudo-Sylvester domains. In the case of crossed products $E * G$ these come, for instance, from free and free-by-cyclic groups \cite{Cohn_Universality,HL22}. More recently, it was shown in \cite{SanchezPeralta_Knots} that locally indicable $3$-manifold groups are Lewin groups.

\subsection{Universal localizations}
    
    Recall that an $R$-module $M$ is of \emph{weak dimension at most $1$} if for every right $R$-module $N$ the module $\Tor_2^R(N, M)$ vanishes. This property is enjoyed by the Linnell ring of torsion-free one-relator groups \cite[Theorem 1.7]{JL23} and free-by-cyclic groups \cite[Lemma 3.7(4)]{HL22}. The weak-dimension bound is particularly useful together  with the vanishing of $\Tor_1$, which is a natural condition in the theory of universal localizations. More specifically, recall that given a set $\Sigma$ of matrices over $R$ there is a unique $R$-ring $R_{\Sigma}$ called \emph{the universal localization of $R$ at $\Sigma$} which is universal with respect to the property of inverting (the image of) each matrix $A$ in $\Sigma$ (see the discussion in \cite{Cohn2006} before Theorem 7.2.4). The construction is as follows: for each $n \times m$ matrix $A = (a_{ij})$ in $\Sigma$ we take a set of $mn$ symbols, arranged as an $m \times n$ matrix $A' = (a'_{ji})$ and let $R_{\Sigma}$ be the ring generated by the elements of $R$ as well as all the $a'_{ji}$ and as defining relations all the relations holding in $R$, together with the relations, in matrix form,
    \[
        A A' = \Id_n, \quad  A'A = \Id_m \quad \mbox{for each $A \in \Sigma$}.
    \]
    By \cite[Theorems 4.7 and 4.8]{Schofield85} $\Tor_1^R(R_{\Sigma}, R_{\Sigma})$ vanishes. Note that for a division $R$-ring $\mathcal{D}$ the latter condition is equivalent to the vanishing of $\beta_1^{\mathcal{D}}(\mathcal{D})$.
    
    \begin{lemma} \label{lem: monotone_b1}
        Let $\mathcal{D}$ be a division $R$-ring and $M$ an $R$-module. Assume that $\mathcal{D}$ has weak dimension at most $1$ as right $R$-module. Then $\beta_n^{\mathcal{D}}(M)$ vanishes for all $n \geq 2$ and $\beta_1^{\mathcal D}(\cdot)$ is a monotonically increasing function on $R$-modules, that is, for any two $R$-modules $M_1 \subseteq M_2$ it holds that
        \[
            \beta_1^{\mathcal{D}}(M_1) \leq \beta_1^{\mathcal{D}}(M_2).
        \]
        In particular, if $\Tor_1^R(\mathcal{D}, \mathcal{D})$ vanishes and $M$ is in addition an $R$-submodule of $\mathcal D$, then
        \[
            \beta_1^{\mathcal{D}}(M) = 0.
        \]
    \end{lemma}
    \begin{proof}
        The fact that $\beta_n^{\mathcal{D}}(M) = 0$ for every $n \geq 2$ follows from the definitions. Next, let $M_1 \subseteq M_2$ be two $R$-modules and consider the natural short exact sequence
        \[
            0 \to M_1 \to M_2 \to M_2/M_1 \to 0.
        \]
        Extending scalars to $\mathcal{D}$ yields the exact sequence
        \[
            \Tor_2^R(\mathcal{D}, M_2/M_1) \to \Tor_1^R(\mathcal{D}, M_1) \to \Tor_1^R(\mathcal{D}, M_2)
        \]
        Since the module on the left vanishes and $\mathcal{D}$ is a division ring, we conclude that $\beta_1^{\mathcal{D}}(M_1) \leq \beta_1^{\mathcal{D}}(M_2)$.
    \end{proof}

\subsection{Pseudo-Sylvester domains}

    A ring $R$ is said to be \emph{stably finite} if whenever $A$ and $B$ are two $n \times n$-matrices over $R$ such that $AB = I_n$, then also $BA = I_n$. We remark that if $E$ is a field of characteristic zero or $E$ is of positive characteristic and $G$ is sofic, then $E G$ is stably finite \cite{Kaplansky72,ElekSzabo04}.

    Let $R$ be a ring, and $A$ an $m \times n$ matrix over $R$. Recall that the \emph{inner rank} $\rho(A)$ is defined as the least $k$ such that $A$ admits a decomposition $A = B_{m \times k} C_{k \times n}$. We say that a square matrix $A$ of size $n \times n$ is \emph{full} if $\rho(A) = n$. Recall also that the \emph{stable rank} $\rho^{*}(A)$ is given by
    \[
        \rho^{*}(A) = \lim_{s \to \infty} (\rho(A \oplus I_s) - s),
    \]
    whenever the limit exists, where $A \oplus I_s$ denotes the block diagonal matrix with blocks $A$ and $I_s$. We analogously say that a square matrix is \emph{stably full} if it has maximum stable rank. When $R$ is stably finite, $\rho^{*}(A)$ is well-defined and non-negative, and it is positive if $A$ is a non-zero matrix \cite[Proposition 0.1.3]{Cohn2006}.
    
    A non-zero ring $R$ is a \emph{pseudo-Sylvester domain} if $R$ is stably finite and satisfies the law of nullity with respect to the stable rank, that is, if $A \in \Mat_{m \times n}(R)$ and $B \in \Mat_{n \times k}(R)$ are such that $A B = 0$, then
    \[
        \rho^{*}(A) + \rho^{*}(B) \leq n.
    \]

    \begin{theorem}[{\cite[Theorem 7.5.18]{Cohn2006}}] \label{thm: univ_pseudoSyl}
        Let $R$ be a pseudo-Sylvester domain. The universal localization of $R$ at the set of stably full matrices is the universal division $R$-ring of fractions.
    \end{theorem}

\section{Groups of type \texorpdfstring{$\FP_2(\ell^2)$}{FP2(l2)}}\label{FP2L2}
Recall that a group \(G\) is said to be of \emph{type \(\FP_{2}(\ell^{2})\)} if there exists a finitely presented group
\(\widetilde{G}\) together with a surjective homomorphism
\(\widetilde{G} \twoheadrightarrow G\) inducing an isomorphism
\[
\mathcal{U}(G)\otimes_{\mathbb{Q}\widetilde{G}} I_{\mathbb{Q}\widetilde{G}}
   \;\xrightarrow{\ \cong\ }\;
\mathcal{U}(G)\otimes_{\mathbb{Q} G} I_{\mathbb{Q} G}.
\]

Note that $G$ is automatically finitely generated. The following proposition gives an equivalent formulation of this
condition.

\begin{proposition}\label{prop:FP2-equivalent}
A group \(G\) is of type \(\FP_{2}(\ell^{2})\) if and only if $G$ is finitely generated and 
\[
\mathcal{U}(G)\otimes_{\mathbb{Q} G} I_{\mathbb{Q} G}
\]
is a finitely presented \(\mathcal{U}(G)\)-module.
\end{proposition}

\begin{proof}
Suppose first that \(G\) is of type \(\FP_{2}(\ell^{2})\).  
By definition, we may choose a finitely presented group  
\(\widetilde{G}\) with a surjection \(\widetilde{G} \twoheadrightarrow G\)  
such that the induced morphism of \(\mathcal{U}(G)\)-modules
\[
\mathcal{U}(G)\otimes_{\mathbb{Q}\widetilde{G}} I_{\mathbb{Q}\widetilde{G}}
   \;\longrightarrow\;
\mathcal{U}(G)\otimes_{\mathbb{Q} G} I_{\mathbb{Q} G}
\]
is an isomorphism.  
Since \(\widetilde{G}\) is finitely presented, the augmentation ideal
\(I_{\mathbb{Q}\widetilde{G}}\) is a finitely presented \(\mathbb{Q}\widetilde{G}\)-module.  
Tensoring with \(\mathcal{U}(G)\) over \(\mathbb{Q}\widetilde{G}\)
preserves finite presentation.  
Hence 
$$
\mathcal{U}(G)\otimes_{\mathbb{Q} G} I_{\mathbb{Q} G}\cong \mathcal{U}(G)\otimes_{\mathbb{Q}\widetilde{G}} I_{\mathbb{Q}\widetilde{G}}
$$
is finitely presented.

Conversely, suppose that
\[
\mathcal{U}(G)\otimes_{\mathbb{Q} G} I_{\mathbb{Q} G}
\]
is finitely presented as a $\mathcal{U}(G)$-module. Let $G=\langle X \mid R\rangle$ be a group presentation with generating set $X= \{x_1,\dots,x_d \}$. The presentation induces an exact sequence of $\mathbb{Q}G$-modules
\[
\bigoplus_{r\in R} \mathbb{Q}G \cdot e_r
\;\longrightarrow\;
\mathbb{Q}G^{\,d}
\xrightarrow{\ \varphi\ }
I_{\mathbb{Q}G}
\;\longrightarrow\; 0,
\]
where
\[
\varphi(a_1,\dots,a_d)=\sum_{i=1}^d a_i(x_i-1).
\]

Applying the right exact functor $\mathcal{U}(G)\otimes_{\mathbb{Q}G} -$ gives an exact sequence of $\mathcal{U}(G)$-modules
\[
\bigoplus_{r\in R} \mathcal{U}(G)\cdot e_r
\;\longrightarrow\;
\mathcal{U}(G)^{\,d}
\xrightarrow{\ \widetilde{\varphi}\ }
\mathcal{U}(G)\otimes_{\mathbb{Q}G} I_{\mathbb{Q}G}
\;\longrightarrow\; 0.
\]

Since $\mathcal{U}(G)\otimes_{\mathbb{Q}G} I_{\mathbb{Q}G}$ is finitely presented, there exists a finite subset $R_0\subseteq R$ such that the truncated sequence
\[
\bigoplus_{r\in R_0} \mathcal{U}(G)\cdot e_r
\;\longrightarrow\;
\mathcal{U}(G)^{\,d}
\xrightarrow{\ \widetilde{\varphi}\ }
\mathcal{U}(G)\otimes_{\mathbb{Q}G} I_{\mathbb{Q}G}
\;\longrightarrow\; 0
\]
remains exact.

Set $\widetilde{G} :=\langle X \mid R_0\rangle$. The natural epimorphism
\[
\widetilde{G} \twoheadrightarrow G
\]
induces an isomorphism
\[
\mathcal{U}(G)\otimes_{\mathbb{Q}\widetilde{G}} I_{\mathbb{Q}\widetilde{G}}
\;\xrightarrow{\ \cong\ }\;
\mathcal{U}(G)\otimes_{\mathbb{Q}G} I_{\mathbb{Q}G}.
\]
This finishes the proof.
\end{proof}

\begin{remark}
    Note that a group $G$ is of type $\FP_2(\ell^2)$ if there exists a group \(\widetilde{G}\) of type $\FP_2(\QQ)$ (not necessarily finitely presented) together with a surjective homomorphism \(\widetilde{G} \twoheadrightarrow G\) inducing an isomorphism
    \[
        \mathcal{U}(G)\otimes_{\mathbb{Q}\widetilde{G}} I_{\mathbb{Q}\widetilde{G}} \;\xrightarrow{\ \cong\ }\; \mathcal{U}(G)\otimes_{\mathbb{Q} G} I_{\mathbb{Q} G}.
    \]
\end{remark}

Using the alternative formulation in \cref{prop:FP2-equivalent}, we show that the following natural classes of groups are of type $\FP_2(\ell^2)$.

\begin{corollary}
    Let $G$ be a finitely generated group. If \(G\) is of type \(\FP_{2}(\mathbb{Q})\), or if \(G\) satisfies the Weak Atiyah Conjecture, then \(G\) is of type \(\FP_{2}(\ell^{2})\).
\end{corollary}

\begin{proof}
If \(G\) is of type \(\FP_{2}(\mathbb{Q})\), then the module
\(  I_{\mathbb{Q}G}\) is finitely presented over
\(\mathbb{Q}G\).  
Tensoring with \(\mathcal{U}(G)\) preserves finite presentation, and therefore
$
\mathcal{U}(G)\otimes_{\mathbb{Q}G} I_{\mathbb{Q}G}
$
is a finitely presented \(\mathcal{U}(G)\)-module.  
Hence \(G\) is of type \(\FP_{2}(\ell^{2})\) by  \cref{prop:FP2-equivalent}.

If \(G\) satisfies the Weak Atiyah Conjecture, then the $*$-regular closure
\(\mathcal{R}_{\mathbb{Q}G}\subset \mathcal{U}(G)\) is a semisimple ring (see \cref{sec: l2_betti_nbrs}).
Consequently, the finitely generated module
\[
\mathcal{R}_{\mathbb{Q}G}\otimes_{\mathbb{Q}G} I_{\mathbb{Q}G}
\]
is finitely presented over \(\mathcal{R}_{\mathbb{Q}G}\).  
Therefore,
\[
\mathcal{U}(G)\otimes_{\mathbb{Q}G} I_{\mathbb{Q}G}
   \;\cong\;
\mathcal{U}(G)\otimes_{\mathcal{R}_{\mathbb{Q}G}}
   \bigl(\mathcal{R}_{\mathbb{Q}G}\otimes_{\mathbb{Q}G} I_{\mathbb{Q}G}\bigr)
\]
is finitely presented over \(\mathcal{U}(G)\).
Thus \(G\) is of type \(\FP_{2}(\ell^{2})\) by  \cref{prop:FP2-equivalent}.
\end{proof}

The following theorem is the main result of this section.  
It subsumes several statements made in the introduction;  
in particular, part \ref{item: equality_b1_l2indep} corresponds to item \ref{item: approach_b1_equality} of \cref{approach}.

\begin{theorem}\label{propfpl2}
    Let $G$ be a group of type $\FP_2(\ell^2)$ with associated map \(\alpha : \widetilde{G} \to G\). Let \(0 \neq \phi \in H^{1}(G;\mathbb{Z})\) and let \(\widetilde{G} = K *_{H,\theta}\) be an admissible splitting dual to $\phi \circ \alpha$. If \(b_{1}^{(2)}(G)=0\), then:
\begin{enumerate}
    \item
    We have the inequalities
    \begin{equation}\label{inequalities}
        b_{1}^{(2)}(\ker \phi)
        \;\le\;
        \dim_{\mathcal{U}(G)}
        \bigl( \mathcal{U}(G)\otimes_{\mathbb{Q} K} I_{\mathbb{Q} K} \bigr)
        -1+\frac{1}{|\alpha(K)|}
        \;\le\;
        \dim_{\mathcal{U}(G)}
        \bigl( \mathcal{U}(G)\otimes_{\mathbb{Q} H} I_{\mathbb{Q} H} \bigr)
        -1+\frac{1}{|\alpha(H)|}.
    \end{equation}

    \item
    If \(\widetilde{G}=G\) and $\alpha$ is the identity, then
    \[
        b_{1}^{(2)}(\ker \phi)
        \;\le\;
        b_{1}^{(2)}(K)
        \;\le\;
        b_{1}^{(2)}(H).
    \]

    \item \label{item: equality_b1_l2indep}
    If \(\widetilde{G}=G\), $\alpha$ is the identity and both \(H\) and $\theta (H)$ are \(L^{2}\)-independent in \(K\), then
    \[
        b_{1}^{(2)}(\ker \phi)
        \;=\;
        b_{1}^{(2)}(K)
        \;=\;
        b_{1}^{(2)}(H).
    \]
\end{enumerate}
\end{theorem}

\begin{proof} Since $\widetilde{G}=K *_{H,\theta}$, we may write
\[
\widetilde{G}
   = \big\langle\, t, K \ \big|\  h t = t\,\theta(h),\; h\in H \,\big\rangle.
\]
The map \(\theta : H \to K\) extends naturally to a homomorphism 
\(\theta : \mathbb{Q}H \to \mathbb{Q}K\).  
In what follows, we will abuse notation and also denote by \(t\) the image of \(t\) in \(G\).

From the splitting of $\widetilde{G}$ we obtain the exact sequence
\[
0 \longrightarrow {}^{\widetilde G} I_{\mathbb{Q}H}
   \xrightarrow{\ \gamma\ }
   {}^{\widetilde G} I_{\mathbb{Q}K} \oplus \mathbb{Q}\widetilde{G}\,(t-1)
   \longrightarrow I_{\mathbb{Q}\widetilde{G}}
   \longrightarrow 0,
\]
where the map $\gamma$ is given on generators by
\[
\gamma(1\otimes a)
   =
   \bigl(1\otimes a - t\otimes\theta(a),\; a(t-1)\bigr)
   \qquad (a\in I_{\QQ H}).
\]

After applying the functor $\mathcal{U}(G)\otimes_{\mathbb{Q}\widetilde G} -$, we obtain the exact sequence
\begin{multline}\label{preseq}
H_2(\widetilde G;\mathcal U(G))  \longrightarrow 
\mathcal{U}(G)\otimes_{\QQ H} I_{\mathbb{Q}H}
   \xrightarrow{\ \widetilde{\gamma}\ }
\mathcal{U}(G)\otimes_{\QQ K} I_{\mathbb{Q}K}
   \oplus
   \mathcal{U}(G)\,(t-1)
   \longrightarrow\\
\mathcal{U}(G)\otimes_{\QQ\widetilde G} I_{\mathbb{Q}\widetilde{G}}
   \longrightarrow 0 ,
\end{multline}
where $\widetilde{\gamma}=\mathrm{id}\otimes \gamma$. 

Now $G$ is of type $\FP_2(\ell^2)$, so $\UO(G) \otimes_{\QQ G} I_{\QQ G}$ is projective by \cref{prop: vNr_proj_mod}. Since $t - 1$ is invertible in $\UO(G)$ we get that the multiplication map
\[
   \mathcal U(G) \otimes_{\QQ G} I_{\QQ G} \to\mathcal U(G)
\]
is surjective. Moreover, since $b_1^{(2)}(G)$ vanishes and $\dim_{\UO(G)}$ is faithful over projective modules, the multiplication map is actually an isomorphism. In particular, the map
\[
    \mathcal U(G)\otimes_{\mathbb{Q}\widetilde{G}} I_{\mathbb{Q}\widetilde{G}} \;\longrightarrow\; \mathcal U(G)
\]
is an isomorphism. Furthermore, observe that the natural inclusion $\UO(G) (t -1) \hookrightarrow\mathcal U(G)$ is an isomorphism (being \(t-1\) is invertible in \(\mathcal U(G)\)) and factors through
\[
    \mathcal U(G)\,(t-1) \;\longrightarrow\;\mathcal U(G)\otimes_{\mathbb{Q}\widetilde{G}} I_{\mathbb{Q}\widetilde{G}},\ t-1\mapsto 1\otimes (t-1).
\]
It follows then that the latter map is also an isomorphism.

Thus, from \cref{preseq} we obtain the exact sequence 
\begin{equation} \label{tildegamma1}
H_2(\widetilde G;\mathcal U(G))  \longrightarrow 
\mathcal{U}(G)\otimes_{\QQ H} I_{\mathbb{Q}H}
   \xrightarrow{\ {\gamma_1}\ }
\mathcal{U}(G)\otimes_{\QQ K} I_{\mathbb{Q}K}
   \longrightarrow 0 
\end{equation}
given by ${\gamma_1}(1\otimes a)=1\otimes a - t\otimes \theta(a)$ for $a\in I_{\QQ H}$.
This implies that
\begin{equation}\label{desig} \dim_{\mathcal{U}(G)}\!
    \left( \mathcal{U}(G)\otimes_{\mathbb{Q} K} I_{\mathbb{Q} K} \right) 
    \;\le\;
    \dim_{\mathcal{U}(G)}\!
    \left( \mathcal{U}(G)\otimes_{\mathbb{Q} H} I_{\mathbb{Q} H} \right).
    \end{equation}
    In particular, if $\alpha(H)$ is finite
    \[
        1 - \frac{1}{|\alpha(K)|}\leq \! \dim_{\mathcal{U}(G)}\!
    \left( \mathcal{U}(G)\otimes_{\mathbb{Q} K} I_{\mathbb{Q} K} \right) \leq \!  \dim_{\mathcal{U}(G)}\!
    \left( \mathcal{U}(G)\otimes_{\mathbb{Q} H} I_{\mathbb{Q} H} \right) = 1 - \frac{1}{\lvert \alpha(H) \rvert},
    \]
    and so, $\alpha(K)$ is also finite and has the same order as $\alpha(H)$. This implies that $\alpha(H)=\alpha(K)$, and so $\ker \phi$ is finite. Therefore,
\[
b_1^{(2)}(\ker \phi)=0,
\]
and \emph{(i), (ii)} and \emph{(iii)} hold if $\alpha(H)$ is finite. 

From now on we assume that $\alpha(H)$ is infinite. Note that in this case the second inequality in \emph{(i)} follows from \cref{desig}.

Let $\widetilde{N}=\ker(\phi\circ \alpha)$.  
For integers $i\le j$, denote by $N_{[i,j]}$ the subgroup of $\widetilde{N}$ generated by the conjugates 
$K^{t^k}$ with $i\le k\le j$.  
Note that $\widetilde{N}$ is the direct limit of the $N_{[i,j]}$, and each $N_{[i,j]}$ is the fundamental group of a line of groups where the vertex groups are $K^{t^k}$ with $i\le k\le j$ and the edge groups are $H^{t^k}$ with $i + 1\le k\le j$.

\begin{claim}\label{conjugate} Let $U$ be a subgroup of $\widetilde G$.
Then for each integer $i$, the map
\[
\tau_i \colon 
\mathcal U(G)\otimes_{\mathbb{Q} U^{t^i}} I_{\mathbb{Q} U^{t^i}}
\;\longrightarrow\;
\mathcal U(G)\otimes_{\mathbb{Q} U} I_{\mathbb{Q} U},
\qquad
\tau_i(u\otimes a)= ut^{-i}\otimes a^{\,t^{-i}},
\]
is an isomorphism of $\mathcal U(G)$-modules.
\end{claim}

\begin{proof}
Define a biadditive map
\[
\pi:\mathcal U(G)\times I_{\mathbb{Q} U^{t^{i}}}
\longrightarrow
\UO(G)\otimes_{\mathbb{Q} U} I_{\mathbb{Q} U},
\qquad
(u,a)\longmapsto u t^{-i}\otimes a^{\,t^{-i}}.
\]
This map satisfies the additional condition that for every \(r\in \mathbb{Q} U^{t^{i}}\),
\[
\pi(ur, a)=u rt^{-i}\otimes a^{\,t^{-i}}=ut^{-i} r^{t^{-i}}\otimes a^{\,t^{-i}}=ut^{-i} \otimes r^{t^{-i}}a^{\,t^{-i}}=ut^{-i} \otimes (ra)^{\,t^{-i}}=\pi(u, ra).
\]
By the universal property of the tensor product, there exists a unique map
\[
\tau_i:\mathcal U(G)\otimes_{\mathbb{Q} U^{t^{i}}} I_{\mathbb{Q} U^{t^{i}}}
\longrightarrow
\mathcal U(G)\otimes_{\mathbb{Q} U} I_{\mathbb{Q} U}
\]
induced by \(\pi\). Similarly, one can define an inverse of $\tau_i$.
\end{proof}
Observe that if \(j_{1}-i_{1}=j_{2}-i_{2}\), then
\[
N_{[i_{1},j_{1}]} = N_{[i_{2},j_{2}]}^{\,t^{\,i_{1}-i_{2}}}.
\]
In particular,
\begin{equation}
\label{iso}
\mathcal U(G)\otimes_{\mathbb{Q} N_{[i_{1},j_{1}]}} I_{\mathbb{Q} N_{[i_{1},j_{1}]}}
\;\cong\;
\mathcal U(G)\otimes_{\mathbb{Q} N_{[i_{2},j_{2}]}} I_{\mathbb{Q} N_{[i_{2},j_{2}]}}.
\end{equation}

For each $n\ge 0$ put $N_n=N_{[-n,0]}$. Thus $N_0\cong K$.
\begin{claim}\label{exactNn}
For each $n \ge 1$ there is an exact sequence of $\mathcal U(G)$-modules
\[
\bigl(\mathcal U(G)\otimes_{\mathbb{Q} H} I_{\mathbb{Q}H}\bigr)^n
\xrightarrow{\ \sigma_n\ }
\bigl(\mathcal U(G)\otimes_{\mathbb{Q} K} I_{\mathbb{Q}K}\bigr)^{\,n+1}
\xrightarrow{\ \rho_n\ }
\mathcal U(G)\otimes_{\mathbb{Q} N_n} I_{\mathbb{Q} N_n}
\longrightarrow 0,
\]
where the map $\sigma_n$ is given on generators by
\[
\sigma_n(1\otimes a_1,\ldots, 1\otimes a_n)
=
\bigl(
-t\otimes \theta(a_1),\;
1\otimes a_1 -t\otimes \theta(a_2),\;
\ldots,\;
1\otimes a_{n-1} -t\otimes \theta(a_n),\;
1\otimes a_n
\bigr)
\]
and the map $\rho_n$ is given on generators by
\[
    \rho_n(1 \otimes b_1, \ldots, 1 \otimes b_n, 1 \otimes b_{n+1}) = \sum_{i = 1}^{n + 1} \tau_{i - n - 1} (1 \otimes b_i).
\]
\end{claim}

\begin{proof} The group \(N_{n}\) is the line of groups
\[
K^{t^{-n}}
\;*_{\,H^{t^{-(n-1)}}}\;
K^{t^{-(n-1)}}
\;\ast_{\,H^{t^{-(n-2)}}}\;
\cdots
\;\ast_{\,H^{t^{-1}}}\;
K^{t^{-1}}
\;*_{\,H}\;
K,
\]
where for each \(0\le i\le n-1\), the embedding
\[
H^{t^{-i}} \longrightarrow K^{t^{-(i+1)}}
\]
is given by \(h^{t^{-i}}\mapsto \theta(h)^{\,t^{-(i+1)}}\), while the embedding
\[
H^{t^{-i}} \longrightarrow K^{t^{-i}}
\]
is the natural inclusion.  
This yields the exact sequence
\[
0 \longrightarrow
{}^{\widetilde G} I_{\mathbb{Q} H^{t^{-(n-1)}}}\oplus \cdots \oplus {}^{\widetilde G} I_{\mathbb{Q} H}
\xrightarrow{\;\eta_{n}\;}
{}^{\widetilde G} I_{\mathbb{Q} K^{t^{-n}}}\oplus \cdots \oplus {}^{\widetilde G} I_{\mathbb{Q} K}
\longrightarrow
{}^{\widetilde G} I_{\mathbb{Q} N_{n}}
\longrightarrow 0,
\]
where
\[
\eta_{n}\!\left(1\otimes a_{1}^{\,t^{-(n-1)}},\ldots,1\otimes a_{n}\right)
=
\bigl(
-1\otimes \theta(a_{1})^{\,t^{-n}},\;
1\otimes a_{1}^{\,t^{-(n-1)}} - 1\otimes \theta(a_{2})^{\,t^{-(n-1)}},\;
\ldots,\;
1\otimes a_{n}
\bigr).
\]

Applying the functor \(\mathcal U(G)\otimes_{\mathbb{Q}\widetilde G} -\), we obtain the exact sequence
\begin{multline}\label{witht}
\mathcal U(G)\otimes_{\mathbb{Q} H^{t^{-(n-1)}}} I_{\mathbb{Q} H^{t^{-(n-1)}}}
\;\oplus\;
\cdots
\;\oplus\;
\mathcal U(G)\otimes_{\mathbb{Q} H} I_{\mathbb{Q} H}
\xrightarrow{\;\widetilde{\eta}_{n}\;}\\
\mathcal U(G)\otimes_{\mathbb{Q} K^{t^{-n}}} I_{\mathbb{Q} K^{t^{-n}}}
\;\oplus\;
\cdots
\;\oplus\;
\mathcal U(G)\otimes_{\mathbb{Q} K} I_{\mathbb{Q} K}
\longrightarrow 
\mathcal U(G)\otimes_{\mathbb{Q} N_{n}} I_{\mathbb{Q} N_{n}}
\longrightarrow 0,    
\end{multline}
where
\[
\widetilde{\eta}_{n}\!\left(1\otimes a_{1}^{\,t^{-(n-1)}},\ldots,1\otimes a_{n}\right)
=
\bigl(
-1\otimes \theta(a_{1})^{\,t^{-n}},\;
1\otimes a_{1}^{\,t^{-(n-1)}} - 1\otimes \theta(a_{2})^{\,t^{-(n-1)}},\;
\ldots,\;
1\otimes a_{n}
\bigr).
\]
Using the isomorphisms established in \cref{conjugate}, the sequence in \cref{witht} transforms into the sequence stated in the claim.
\end{proof}
\begin{claim}\label{subD}
Let
\[
D=\{\, (u,-u) \;\mid\; u \in \mathcal U(G)\otimes_{\mathbb{Q} K} I_{\mathbb{Q}K} \,\}
   \subset 
\bigl(\mathcal U(G)\otimes_{\mathbb{Q} K} I_{\mathbb{Q}K}\bigr)^{2}.
\]
Then
\[
\rho_1(D)=\mathcal U(G)\otimes_{\mathbb{Q} N_1} I_{\mathbb{Q} N_1}.
\]
\end{claim}

\begin{proof}
It suffices to show that 
\[
D + \im\sigma_{1}
= \bigl(\mathcal U(G)\otimes_{\mathbb{Q} K} I_{\mathbb{Q}K}\bigr)^{2}.
\]
Observe that \(D + \im \sigma_{1}\) contains the submodule \((0,\, \im  {\gamma}_{1})\). 
However, as established in \cref{tildegamma1}, we have 
\(\im {\gamma}_{1}
= \mathcal U(G)\otimes_{\mathbb{Q} K} I_{\mathbb{Q}K}\).
This proves the claim.
\end{proof}

\begin{claim}\label{bound}
For all $i\le j$ we have
\[
\dim_{\mathcal U(G)}
\bigl( \mathcal U(G)\otimes_{\mathbb{Q} N_{[i,j]}}
   I_{\mathbb{Q} N_{[i,j]}} \bigr)
\;\le\;
\dim_{\mathcal U(G)}
\bigl( \mathcal U(G)\otimes_{\mathbb{Q} K} I_{\mathbb{Q} K} \bigr).
\]
\end{claim}

\begin{proof} By \cref{iso}, we may assume without loss of generality that \(j=0\).
We proceed by induction on \(i\).

If \(i=0\), then \(N_{0}=K\), and the claim follows immediately. If $i= - 1$, the claim follows from \cref{subD}.
Assume now that the claim holds for \(i=1-n\le -1\), and consider the case \(i=-n\le -2\).

For each \(1\le k\le n\), let
\[
\mu_{k} : 
\mathcal U(G)\otimes_{\mathbb{Q} K} I_{\mathbb{Q}K}
\longrightarrow 
\bigl(\mathcal U(G)\otimes_{\mathbb{Q} K} I_{\mathbb{Q}K}\bigr)^{\,n+1}
\]
denote the map sending an element \(r\) to the vector whose \(k\)-th coordinate is \(r\), whose \((k+1)\)-st coordinate is \(-r\), and whose remaining coordinates are \(0\).  
Set \(D_{k}=\im\mu_{k}\).

By \cref{subD}, we have
\begin{equation}
\label{sumDi}
\rho_{n}(D_{1}+\cdots + D_{n})
=
\mathcal U(G)\otimes_{\mathbb{Q} N_{n}} I_{\mathbb{Q} N_{n}}.
\end{equation}
Fix \(a\in I_{\mathbb{Q} H}\) and $1\le k\le n-1$ and consider the element 
\[
m \;=\; \mu_{k}\bigl( 1\otimes a\bigr) -\mu_{k + 1}\bigl(t\otimes \theta(a)\bigr)\in D_k+D_{k+1}.
\]
By \cref{exactNn}, we have 
\begin{multline}
\label{kernel}
\rho_{n}(m)
= \rho_n\!\bigl(\mu_{k}(1\otimes a) - \mu_{k+1}(t\otimes \theta(a))\bigr) \\
= \tau_{k-1-n}(1\otimes a)
  - \tau_{k-n}\bigl(1\otimes a + t\otimes \theta(a)\bigr)
  + \tau_{k+1-n}(t\otimes \theta(a)) \\
= t^{n+1-k}\otimes a^{t^{n+1-k}}
  - t^{n-k}\otimes a^{t^{n-k}}
  - t^{n+1-k}\otimes \theta(a)^{t^{n-k}}
  + t^{n-k}\otimes \theta(a)^{t^{n-k-1}} \\
= t^{n+1-k}\otimes a^{t^{n+1-k}} - t^{n-k}\otimes a^{t^{n-k}} - t^{n+1-k}\otimes a^{t^{n+1-k}} + t^{n-k}\otimes a^{t^{n-k}} 
= 0.
\end{multline}
Define
\[
\mu:\bigl(\mathcal U(G)\otimes_{\mathbb{Q}K} I_{\mathbb{Q}K}\bigr)^{n}
\longrightarrow D_1+\cdots+D_n
\]
by
\[
\mu(m_1,\ldots,m_n)= \sum_{i=1}^n \mu_{n + 1 -i}(m_i).
\]
Then, combining \cref{exactNn},  \cref{kernel} and  \cref{sumDi}, we obtain the following commutative diagram:

\[
\begin{tikzcd}
\bigl(\mathcal U(G)\otimes_{\mathbb{Q}H} I_{\mathbb{Q}H}\bigr)^{n-1}
\arrow[r, "\sigma_{n-1}"]
\arrow[d, "\mu \circ \sigma_{n-1}"']
&
\bigl(\mathcal U(G)\otimes_{\mathbb{Q}K} I_{\mathbb{Q}K}\bigr)^{n}
\arrow[r, "\rho_{n-1}"]
\arrow[d, "\mu"]
&
\mathcal U(G)\otimes_{\mathbb{Q}N_{n-1}} I_{\mathbb{Q}N_{n-1}}
\arrow[r]
&
0 \\
(D_1+\cdots+D_n)\cap \ker \rho_n
\arrow[r]
&
D_1+\cdots+D_n
\arrow[r, "\rho_n"]
&
\mathcal U(G)\otimes_{\mathbb{Q}N_n} I_{\mathbb{Q}N_n}
\arrow[r]
&
0
\end{tikzcd}.
\]
Since   $\mu$ is surjective,  we  deduce that
\[
\mathcal U(G)\otimes_{\mathbb{Q} N_n} I_{\mathbb{Q} N_n}
\]
is a quotient of
\[
\mathcal U(G)\otimes_{\mathbb{Q} N_{n-1}} I_{\mathbb{Q} N_{n-1}}.
\]
Applying the inductive hypothesis completes the inductive step and therefore the proof of the claim.
\end{proof}

Since \(\widetilde N\) is the direct limit of the subgroups \(N_{[i,j]}\), the module
\(\mathcal{U}(G)\otimes_{\mathbb{Q}\widetilde N} I_{\mathbb{Q}\widetilde N}\) is the direct limit of the modules
\(\mathcal{U}(G)\otimes_{\mathbb{Q} N_{[i,j]}} I_{\mathbb{Q} N_{[i,j]}}\). Hence by \cref{lem: dim_direct_lim}
\[
    \dim_{\mathcal{U}(G)}
    \bigl(\mathcal{U}(G)\otimes_{\mathbb{Q}\widetilde N} I_{\mathbb{Q}\widetilde N}\bigr)
    \;\le\;
    \sup_{i,j}\,
    \dim_{\mathcal{U}(G)}
    \bigl(\mathcal{U}(G)\otimes_{\mathbb{Q} N_{[i,j]}} I_{\mathbb{Q} N_{[i,j]}}\bigr).
\]
Applying \cref{bound}, we obtain
\[
    \dim_{\mathcal{U}(G)}
    \bigl(\mathcal{U}(G)\otimes_{\mathbb{Q}\widetilde N} I_{\mathbb{Q}\widetilde N}\bigr)
    \;\le\;
    \dim_{\mathcal{U}(G)}
    \bigl(\mathcal{U}(G)\otimes_{\mathbb{Q} K} I_{\mathbb{Q} K}\bigr).
\]

Let \(N=\ker \phi\).  
Since \(\mathcal{U}(G)\otimes_{\mathbb{Q} N} I_{\mathbb{Q} N}\) is a quotient of 
\(\mathcal{U}(G)\otimes_{\mathbb{Q}\widetilde N} I_{\mathbb{Q}\widetilde N}\), it follows that
\begin{multline*}
    b_{1}^{(2)}(N)
    \;=\;
    \dim_{\mathcal{U}(G)}
    \bigl(\mathcal{U}(G)\otimes_{\mathbb{Q} N} I_{\mathbb{Q} N}\bigr)
    -1
    \;\le\;\\
    \dim_{\mathcal{U}(G)}
    \bigl(\mathcal{U}(G)\otimes_{\mathbb{Q}\widetilde N} I_{\mathbb{Q}\widetilde N}\bigr)-1
    \;\le\;
    \dim_{\mathcal{U}(G)}
    \bigl(\mathcal{U}(G)\otimes_{\mathbb{Q} K} I_{\mathbb{Q} K}\bigr)-1.
\end{multline*}
This establishes the first inequality in \emph{(i)}.

The two inequalities in \emph{(ii)} follow directly from the corresponding inequalities in \emph{(i)}.

Now assume that $\widetilde{G}=G$ and that \(H\) is \(L^{2}\)-independent in \(K\).  
This implies that $$b_{1}^{(2)}(H)\le b_{1}^{(2)}(K).$$  
Since we have already established the opposite inequality, we conclude that
\[
    b_{1}^{(2)}(H)= b_{1}^{(2)}(K).
\]

In addition, assume that $\theta(H)$ is \(L^{2}\)-independent in \(K\).  

\begin{claim}\label{UG-iso}
    Let \(i_{2}\le i_{1}\le j_{1}\le j_{2}\). Then the natural map
    \[
        \mathcal U(G)\otimes_{\mathbb{Q} N_{[i_{1},j_{1}]}}
        I_{\mathbb{Q} N_{[i_{1},j_{1}]}}
        \;\longrightarrow\;
        \mathcal U(G)\otimes_{\mathbb{Q} N_{[i_{2},j_{2}]}}
        I_{\mathbb{Q} N_{[i_{2},j_{2}]}}
    \]
    is a $\dim_{\mathcal U(G)}$-isomorphism.
\end{claim}

\begin{proof}
    It suffices to treat the cases
    \[
        (i_{2},j_{2})=(i_{1}-1,\, j_{1})
        \qquad\text{and}\qquad
        (i_{2},j_{2})=(i_{1},\, j_{1}+1).
    \]
    We handle only the first case; the second is analogous.

    By \cref{conjugate}, we may assume without loss of generality that
    \(j_{1}=j_{2}=0\).  Set \(n=-i_{2}\).
 From    the sequence in \cref{exactNn}  we obtain an exact sequence
    \[
        \mathcal U(G)\otimes_{\mathbb{Q} H} I_{\mathbb{Q} H}
        \xrightarrow{\;\epsilon\;}
        \mathcal U(G)\otimes_{\mathbb{Q} K} I_{\mathbb{Q} K}
        \;\oplus\;
        \mathcal U(G)\otimes_{\mathbb{Q} N_{n-1}} I_{\mathbb{Q} N_{n-1}}
        \longrightarrow
        \mathcal U(G)\otimes_{\mathbb{Q} N_{n}} I_{\mathbb{Q} N_{n}}
        \longrightarrow 0,
    \]
    where \(\epsilon=(\epsilon_{1},\epsilon_{2})\) and
    \(\epsilon_1(1\otimes a)=t\otimes \theta(a)\).

    Since \(\theta(H)\) is \(L^{2}\)-independent in \(K\) and \emph{(i)} holds, the homomorphism
    \(\theta : H\to K\) induces a $\dim_{\mathcal U(G)}$-isomorphism
    \[
      \widetilde \theta:  \mathcal U(G)\otimes_{\mathbb{Q} H} I_{\mathbb{Q} H}
        \;\longrightarrow\;
        \mathcal U(G)\otimes_{\mathbb{Q} K} I_{\mathbb{Q} K}.
    \]
    Since $\im \epsilon_1=\im \widetilde \theta$, the homomorphism \(\epsilon_{1}\) is also a $\dim_{\mathcal U(G)}$-isomorphism.
    Therefore, this implies that the map
    \[
        \mathcal U(G)\otimes_{\mathbb{Q} N_{n-1}} I_{\mathbb{Q} N_{n-1}}
        \;\longrightarrow\;
        \mathcal U(G)\otimes_{\mathbb{Q} N_{n}} I_{\mathbb{Q} N_{n}}
    \]
    is a $\dim_{\mathcal U(G)}$-isomorphism as well.
\end{proof}

Since \(N\) is the direct limit of the subgroups \(N_{[i,j]}\), the module
\(\mathcal{U}(G)\otimes_{\mathbb{Q}N} I_{\mathbb{Q}N}\) is the direct limit of the modules
\(\mathcal{U}(G)\otimes_{\mathbb{Q} N_{[i,j]}} I_{\mathbb{Q} N_{[i,j]}}\).
By \cref{UG-iso} along with \cref{lem: comp_dim_isomph}, the natural map
\[
    \mathcal{U}(G)\otimes_{\mathbb{Q} K} I_{\mathbb{Q} K}
    \;\longrightarrow\;
    \mathcal{U}(G)\otimes_{\mathbb{Q} N} I_{\mathbb{Q} N}
\]
is a $\dim_{\mathcal{U}(G)}$-isomorphism.  
Consequently,
\[
    b_{1}^{(2)}(N)=b_{1}^{(2)}(K),
\]
which completes the proof of \emph{(iii)}.
\end{proof}

We conclude this section with the proof of item \ref{item: approach_l2indep_splitting} of \cref{approach}.
\begin{proof}[Proof of \cref{approach}(ii)]
    Since \(G\) satisfies the Weak Atiyah Conjecture, there exists a splitting  
\(G = K \ast_{H,\theta}\) that realises \(c(G;\phi)\).  
Suppose that $H$ is not \(L^{2}\)-independent in $K$.  
Then, by our assumption, $H$ is not \(L^{2}\)-compressed in $K$.  
Hence there exists a finitely generated subgroup $H'$ of $K$ containing \(H\) such that 
\[
    b_{1}^{(2)}(H') < b_{1}^{(2)}(H).
\]

Write the original
HNN extension as
\[
G=\left\langle K,t\ \middle|\ t^{-1}ht=\theta(h),\ h\in H\right\rangle.
\]
 
Put
\[
 \overline{H'}:=t^{-1}H't
\]
and, for \(h'\in H'\), write
\(\overline{h'}=t^{-1}h't\). Define
\[
 K':=\langle K,\overline{H'}\rangle
     =\langle K,t^{-1}H't\rangle\leq \ker\phi,
 \qquad
 \theta'\colon H'\longrightarrow K',
 \quad
 \theta'(h')=\overline{h'}.
\]
Since both \(K\) and \(H'\) are finitely generated, so is \(K'\).
Consider the abstract HNN extension
\[
 G':=K'\ast_{H',\theta'}
   =\left\langle K',s\ \middle|\
       s^{-1}h's=\theta'(h'),\ h'\in H'\right\rangle.
\]

There is a homomorphism
\[
 p\colon G'\longrightarrow G,
 \qquad
 p(k')=k'\quad(k'\in K'),
 \qquad
 p(s)=t.
\]
It is well defined because, for every \(h'\in H'\),
\[
 p(s)^{-1}p(h')p(s)
   =t^{-1}h't
   =\theta'(h')
   =p\bigl(\theta'(h')\bigr).
\]
Moreover, \(p\) is surjective, since \(K\leq K'\) and
\(G=\langle K,t\rangle\).

Conversely, define
\[
 q\colon G\longrightarrow G',
 \qquad
 q(k)=k\quad(k\in K),
 \qquad
 q(t)=s.
\]
To see that \(q\) is well defined, let \(h\in H\). Since
\(H\leq H'\), the defining relation of \(G'\) gives
\[
 s^{-1}hs=\theta'(h).
\]
On the other hand, in the concrete subgroup \(K'\leq G\), we have
\[
 \theta'(h)=t^{-1}ht=\theta(h).
\]
Consequently,
\[
 s^{-1}hs=\theta(h),
\]
so the defining relations of the original HNN presentation of \(G\)
are satisfied in \(G'\).

The homomorphisms \(p\) and \(q\) are inverse. Indeed,
\(p\circ q\) fixes \(K\) and \(t\), and hence
\[
 p\circ q=\operatorname{id}_G.
\]
The group \(G'\) is generated by \(s\), by \(K\), and by the elements
\(\overline{h'}\), where \(h'\in H'\). The map \(q\circ p\) fixes
\(s\) and \(K\), while
\[
 \begin{aligned}
 (q\circ p)(\overline{h'})
   &=q(t^{-1}h't)\\
   &=s^{-1}h's\\
   &=\theta'(h')\\
   &=\overline{h'}.
 \end{aligned}
\]
Therefore,
\[
 q\circ p=\operatorname{id}_{G'},
\]
and hence \(G'\cong G\).

The base group \(K'\) is contained in \(\ker\phi\), and under the
above isomorphism the stable letter \(s\) is sent to \(t\).
Consequently,
\[
 G\cong K'\ast_{H',\theta'}
\]
is an admissible splitting dual to \(\phi\). Its complexity is
\(b_1^{(2)}(H')\), and
\[
 b_1^{(2)}(H')
   <b_1^{(2)}(H)
   =c(G;\phi),
\]
contradicting the choice of the original splitting. Thus \(H\) is
\(L^2\)-independent in \(K\).
 
A similar argument shows that $\theta(H)$ is also \(L^{2}\)-independent in $K$.
\end{proof}

\section{The convexity of the Thurston map}\label{sec:thurston}
In this section we prove \cref{thurston} and its generalization \cref{higher}.  
Both results will follow from  \cref{thm: Thurston_norm_modules,existencepolytope}.

\subsection{Polytopes}\label{sec:polytopes}
Let $H$ be  a free abelian group of finite rank. A \emph{polytope} is the convex hull of a finite set of points in $\mathbb{R}\otimes H$. In particular, the empty set is also a polytope.

Given a non-empty convex compact set $P$ and a character $\phi \in H^1(H; \mathbb{R})$, we define a \emph{face associated with $\phi$} as 
\[
F_{\phi}(P) := \left\{ \, p \in P \,\middle|\, \phi(p) = \min_{q \in P} \phi(q) \, \right\}.
\]
 A face is called a \emph{vertex} if and only if it has dimension~0. Given a face $F$ of the polytope $P$, we define the \emph{dual region in $H^1(H;\mathbb{R})$ to $F$} as
\[ D_F(P) := \left\{ \phi \in H^1(H; \RR) \mid F_{\phi}(P) = F\right\}.\]
Note that if $v$ is a vertex of $P$ then $D_v(P)$ is an open subset of $H^1(H;\RR)$.

 We denote by $\overline{P}$ the set symmetric to $P$ with respect to the origin, that is
\[
\overline{P} = \{\, -p \mid p \in P \,\}.
\]
We will say that $P$ is \emph{symmetric} if $P=\overline P$.

Let \(P\subset \mathbb{R}\otimes H\) be a nonempty  closed convex set with \(0\in P\).
Define the \emph{dual} of \(P\) by
\[
P^\ast :=\{\,\phi\in H^1(H;\RR)\mid \phi(p)\le 1 \text{ for all } p\in P\,\}.
\]

Symmetrically, if  \(B\subset H^1(H;\RR)\) is a nonempty  closed  convex set with \(0\in  B\), we define the \emph{dual} of \(B\) by
\[
B^\ast :=\{\,x\in \mathbb{R}\otimes H \mid \phi(x)\le 1 \text{ for all } \phi \in B\,\}.
\]
\begin{lemma}\label{dual}
The dual of \(P^\ast\) and $B^{*}$ equals \(P\) and $B$ respectively, that is
\[
    \bigl(P^\ast\bigr)^\ast = P \quad \mbox{and} \quad \bigl(B^\ast\bigr)^\ast = B.
\]
\end{lemma}

\begin{proof}
First note that \(P^\ast\) is closed and convex.
By the definition of the dual we immediately have \(P\subseteq (P^\ast)^\ast\): if \(p\in P\) and \(\phi\in P^\ast\) then \(\phi(p)\le 1\), hence \(p\) satisfies the defining inequalities of \((P^\ast)^\ast\).

For the reverse inclusion suppose \(x\notin P\). Since \(P\) is closed and convex and \(x\notin P\), the (finite-dimensional) hyperplane separation theorem (see for example, \cite[Theorem 7]{Ti87}) provides      \(\phi \in H^1(H;\RR)\) and a real number \(c\) such that
\[
\phi(x)> c \quad\text{and}\quad \phi(p) < c \quad\text{for all } p\in P.
\]
Because \(0\in P\), we have \(c > 0\). Set \(\psi   := \phi/c\). Then for all \(p\in P\) we have \(\psi( p) < 1\), so \(\psi\in P^\ast\). But
\[
\psi(x)  = \frac{\phi(x)}{c} > 1,
\]
so \(x\) does not satisfy the inequalities defining \((P^\ast)^\ast\). Hence \(x\notin (P^\ast)^\ast\). Combining the two inclusions yields \((P^\ast)^\ast = P\), as required.

A symmetric argument shows the other equality.
\end{proof}

The \emph{Minkowski sum} of two convex non-empty closed sets $P$ and $Q$ is defined by
\[
P + Q := \{\, p + q \mid p \in P,\ q \in Q \,\}.
\]
A polytope $P$ is called \emph{integral} if and only if all its vertices lie in $H \subseteq H_1(H; \mathbb{R})$.  
We define $\mathcal P(H)$ to be the Grothendieck group of the monoid of non-empty integral polytopes in $H_1(H; \mathbb{R})$. The group $\mathcal P(H)$ is constructed from the free abelian group with basis given by the set of non-empty integral polytopes, modulo the relations induced by the Minkowski sum. It is easy to see that every element in $\mathcal{P}(H)$ can be represented as a formal difference $P - Q$ of integral polytopes. We say that an element is a \emph{single polytope} if and only if we can take $Q$ to be a singleton. We also define $\mathcal P_T (H )$ to be the quotient of $\mathcal P(H )$ by the subgroup consisting of  singletons. The class of $U\in \mathcal P (H )$ in $\mathcal P_T (H )$ is denoted by $[U]$.

For each non-empty convex compact set $P$, define the thickness function $\gr_P \colon H^1(H; \mathbb{R}) \to \mathbb{R}_{\ge 0}$ by
\[
\gr_P(\phi)
    := \max_{p \in P} \phi(p) \;-\; \min_{p \in P} \phi(p).
\]
Observe that $\gr_{P+Q} = \gr_P + \gr_Q$.  
Therefore, for any $U \in \mathcal{P}(H)$ represented as $U = P - Q$, where $P$ and $Q$ are non-empty polytopes, 
we define
\[
\gr_U := \gr_P - \gr_Q.
\]
Observe that $\gr$ factors through $\mathcal P_T(H)$. Thus, we will also write $\gr_{[U]}$ for $[U]\in \mathcal P_T(H)$.
\begin{lemma}\label{thickfunction}
Let $P$ and $Q$ be two compact, convex and symmetric sets. If $\gr_P = \gr_Q$, then $P = Q$.
\end{lemma}

\begin{proof}
Observe that we have the following equality:
\[
P^*=\{\phi \in H^1(H; \RR) \mid \gr_P(\phi)\le 2\}.
\]
Thus, $P^*=Q^*$.
Therefore, by  \cref{dual}, $P=Q$.
\end{proof}

\subsection{The Thurston norm and the Newton polytope of a module} 
Throughout this section let $\mathcal{E}$ be a division ring and $H$ a free abelian group of finite rank. Let $S$ denote the crossed product $\mathcal{E} * H$, and let $\mathcal{D}$ be the Ore ring of fractions of $S$. Consider the category $\mathcal{T}_S$ consisting of all finitely generated $S$-modules $N$ such that
\[
\mathcal{D} \otimes_S N = \{0\}.
\]
In this section we will define the Thurston norm $x_N$ and the Newton polytope $P(N) \in \mathcal{P}(H)$ of a module $N \in \mathcal{T}_S$ and show that $P(N)+\overline{P(N)}$ can be represented by a  single symmetric integral polytope.

Let $s$ be a non-zero element of $S$. The \emph{Newton polytope} of $s$, denoted $P(s)$, is the convex hull of $\operatorname{supp}(s)$,
considered inside the $\mathbb{R}$-vector space $\mathbb{R} \otimes H$.  
By \cite[Lemma~5.16]{Kielak2020polytopes}, we have
\[
P(s_1 s_2) = P(s_1) + P(s_2).
\]
Thus, for any $d \in \mathcal{D}^{\times}$ we may write $d = s_1 s_2^{-1}$ with $s_1, s_2 \in S \setminus \{0\}$,  
and define
\[
P(d) := P(s_1) - P(s_2).
\]

We recall that for a ring $R$, the \emph{projective class group} $K_0(R)$ is the free abelian group on finitely generated projective $R$-modules modulo the relation $[L] = [L_1] + [L_2]$ if there is a short exact sequence of $R$-modules $0 \to L_1 \to L \to L_2 \to 0$. Here $[L]$ denotes the element of $K_0(R)$ corresponding to the projective $R$-module $L$. One can define $K_0(\mathcal{T}_S)$ similarly over the $S$-modules in $\mathcal{T}_S$. The $K_1$-group of $R$ is defined as the abelianisation of the infinite general linear group over $R$, that is, $K_1(R):= \GL(R)^{\text{ab}}$. 

In \cite[Theorem~5]{Qu72} Quillen showed that there exists the following long exact sequence in algebraic $K$-theory:
\[
\cdots \longrightarrow K_1(S) \longrightarrow K_1(\mathcal{D}) \longrightarrow K_0(\mathcal{T}_S) \longrightarrow K_0(S) \longrightarrow K_0(\mathcal{D}) \longrightarrow 0.
\]
Observe that $K_0(\mathcal{D}) \cong \mathbb{Z}$, and by \cite{Moody89}, the map $K_0(S) \to K_0(\mathcal{D})$ is bijective. Hence, the map
\[
K_1(\mathcal{D}) \longrightarrow K_0(\mathcal{T}_S)
\]
is surjective. Therefore, given $W \in K_0(\mathcal{T}_S)$, there exists an element
\[
d \in K_1(\mathcal{D})
\cong
\mathcal{D}^{\times}/[\mathcal{D}^{\times},\mathcal{D}^{\times}]
\]
that maps to $W$, and we define $P(W) \in \mathcal P_T(H)$ to be the image of $P(d)$ in $\mathcal P_T(H)$. This definition is well defined since the Newton polytope of any element of $K_1(S)$ is a singleton (see the argument in the proof of \cite[Corollary~5.16]{Kielak2020polytopes}). Similarly, given $N \in \mathcal{T}_S$ we define 
\[
    P(N) := P([N]).
\]
Observe that if $s\in S \setminus \{0\}$, then $P(S/Ss)$ is the image of $P(s)$ in $\mathcal P_T(H)$.

In \cite{Kielak2020polytopes} Kielak proved the following theorem.
\begin{theorem}[{\cite[Theorem 3.14]{Kielak2020polytopes}}] \label{kielak}
    Let $A\in \Mat_k(S)$ be a square matrix over $S$ that is invertible over $\mathcal D$. Then 
    $P(S^k/S^k A)$ is a single integral polytope.
\end{theorem}
The purpose of this section is to prove the following variation of the previous theorem.
\begin{theorem}\label{existencepolytope} 
Let $N\in \mathcal T_S$. Then $P(N)+\overline{P(N)}=[P]$, where $P$ is a single symmetric integral polytope. 
\end{theorem}

The proof of \cref{existencepolytope}, which is given after
\cref{thm: Thurston_norm_modules}, combines three ingredients developed below.
First, \cref{thick} identifies the Thurston map \(x_{W}\) with the thickness
function \(\gr_{P(W)}\) of the Newton-polytope class. Second, \cref{pd1}
shows that a torsion module over a rank-two crossed product, provided that it
has no non-trivial finite-dimensional submodules, admits a square presentation;
hence \cref{kielak} applies and its Newton-polytope class is represented by a
single polytope. Third, \cref{thm: Thurston_norm_modules} reduces an arbitrary
pair of characters to this rank-two situation and proves the subadditivity of
\(x_N\), so that \(x_N\) extends to a seminorm and convex duality produces a
centrally symmetric compact convex set with thickness
function \(2x_N\). The symmetrization
is essential because thickness
functions are invariant under translations and
reflection, and satisfy
\(\gr_{P+\overline{P}}=2\gr_P\). Consequently,
\(x_N=\gr_{P(N)}\) canonically detects the symmetrized class
\(P(N)+\overline{P(N)}\), but does not in general retain enough one-sided
support data to recover the unsymmetrized class \(P(N)\).

 We start the proof by  introducing the Thurston map of $N$. For each character $\phi \colon H \to \mathbb{Z}$, define
\[
R_{\phi} := \mathcal{E} * \ker(\phi),
\]
and let $\mathcal{D}_{\phi}$ denote its Ore localization. Note that $\mathcal{D}_{\phi}$ is naturally a subring of $\mathcal{D}$ and let $T_{\phi}$ be the subring of $\mathcal{D}$ generated by $\mathcal{D}_{\phi}$ and $S$. Observe that $T_{\phi}$ has a crossed product structure
\[
T_{\phi} \cong \mathcal{D}_{\phi} * \left(H / \ker \phi\right).
\]
Observe also that $T_\phi$ is an Ore localization of $S$: indeed, it suffices to invert the nonzero elements of $R_\phi$. Consequently, $T_\phi$ is flat as an $S$-module.

For each $N \in \mathcal{T}_S$, the \emph{Thurston map of $N$}, denoted $x_N$, is the function on $H^1(H; \mathbb{Q})$ that assigns to each $\phi \in H^1(H; \mathbb{Q})$ the value
\[
    x_N(\phi) := \alpha_\phi\cdot \dim_{\mathcal{D}_{\phi}}\!\bigl(T_{\phi} \otimes_S N\bigr)
\]
where $\alpha_{\phi}$ is the unique non-negative rational number such that $\phi(H)=\alpha_\phi \Z$. Since any arbitrary element in $K_0(\mathcal{T}_S)$ can be represented as $[S/Ss_1] - [S/Ss_2]$ for some non-trivial elements $s_1, s_2 \in S$, we can also define the Thurston map $x_{W}$ for every $W \in K_0(\mathcal{T}_S)$. Indeed, observe that if $0 \to N_1 \to N_2 \to N_3 \to 0$ is an exact sequence, then $x_{N_2} = x_{N_1} + x_{N_3}$, because $T_\phi$ is flat as an $S$-module.

A natural question is whether the Thurston map of $N \in \mathcal{T}_S$ is subadditive, that is
\[
    x_{N}(\phi_1 + \phi_2) \leq x_N(\phi_1) + x_N(\phi_2) \quad \mbox{for $\phi_1, \phi_2 \in H^1(H; \mathbb{Q})$.}
\]

We will show in \cref{thm: Thurston_norm_modules} that this is the case. 

\begin{lemma}\label{thick}
For each $W \in K_0(\mathcal T_S)$ and $\phi \in H^1(H; \mathbb{Q})$ we have
\[
x_{W}(\phi) = \gr_{P(W)}(\phi).
\]
\end{lemma}

\begin{proof}

When $\phi$ is trivial chracter, the statement is clear. We assume that $\phi$ is not trivial.
Since \(T_\phi\) is an Ore localization of \(S\), it is flat as a right
\(S\)-module. Consequently, the assignment
\[
W\longmapsto x_{W}(\phi)
\]
is additive on \(K_0(\mathcal T_S)\). The assignment
\[
W\longmapsto \gr_{P(W)}(\phi)
\]
is also additive, by the definition of \(P(W)\) and the identity
\(\gr_{P+Q}=\gr_P+\gr_Q\). Since every element of
\(K_0(\mathcal T_S)\) can be written as a difference of classes of the
form \([S/Ss]\), where \(0\neq s\in S\), it is enough to prove the
statement for
\[
W=[S/Ss].
\]

Choose \(h_\phi\in H\) such that
\[
\phi(h_\phi)=\alpha_\phi.
\]
Then
\[
H=\ker(\phi)\oplus \langle h_\phi\rangle.
\]
Let \(z\) be a homogeneous unit of \(S\) corresponding to \(h_\phi\).
Conjugation by \(z\) induces an automorphism
\(\sigma\in\operatorname{Aut}(\mathcal D_\phi)\), and we obtain an
identification
\[
T_\phi\cong \mathcal D_\phi[z^{\pm1};\sigma],
\qquad
za=\sigma(a)z
\quad\text{for every }a\in\mathcal D_\phi.
\]

For a nonzero element
\[
f=\sum_{i=p}^{q} b_i z^i\in T_\phi,
\qquad b_p,b_q\neq 0,
\]
set
\[
\deg_-(f)=p,
\qquad
\deg_+(f)=q.
\]
If \(f,g\in T_\phi\setminus\{0\}\), then the extremal terms in the
product \(fg\) cannot cancel.  
In particular,
\[
\deg_+(fg)-\deg_-(fg)
=
\bigl(\deg_+(f)-\deg_-(f)\bigr)
+
\bigl(\deg_+(g)-\deg_-(g)\bigr).
\]

Write
\[
s=\sum_{i=m}^{n}a_i z^i,
\qquad
a_m,a_n\neq 0.
\]
We claim that
\[
\dim_{\mathcal D_\phi}(T_\phi/T_\phi s)=n-m.
\]
If \(m=n\), then \(s=a_mz^m\) is a unit in \(T_\phi\), so both sides
are zero. Suppose, therefore, that \(m<n\).

We first show that the images of
\[
z^m,z^{m+1},\ldots,z^{n-1}
\]
span \(T_\phi/T_\phi s\) over \(\mathcal D_\phi\). For every \(k\geq 0\),
the relation
\[
z^ks
=
\sum_{i=m}^{n}\sigma^k(a_i)z^{i+k}
\in T_\phi s
\]
gives
\[
z^{n+k}
\equiv
-\sigma^k(a_n)^{-1}
\sum_{i=m}^{n-1}\sigma^k(a_i)z^{i+k}
\pmod{T_\phi s}.
\]
An induction on \(k\) therefore reduces every power \(z^j\) with
\(j\geq n\) to a \(\mathcal D_\phi\)-linear combination of
\(z^m,\ldots,z^{n-1}\). Similarly, for every \(k\geq 1\), the relation
\[
z^{-k}s
=
\sum_{i=m}^{n}\sigma^{-k}(a_i)z^{i-k}
\in T_\phi s
\]
gives
\[
z^{m-k}
\equiv
-\sigma^{-k}(a_m)^{-1}
\sum_{i=m+1}^{n}\sigma^{-k}(a_i)z^{i-k}
\pmod{T_\phi s}.
\]
An induction on \(k\) reduces every power \(z^j\) with \(j<m\) to the
same spanning set.

It remains to prove linear independence. Suppose that
\[
0\neq q=\sum_{i=m}^{n-1}b_i z^i\in T_\phi s.
\]
Then \(q=rs\) for some nonzero \(r\in T_\phi\). By additivity of the
extremal degrees,
\[
\deg_+(q)-\deg_-(q)
=
\deg_+(r)-\deg_-(r)+n-m
\geq n-m.
\]
On the other hand, since \(q\) is supported in
\(\{m,\ldots,n-1\}\), we have
\[
\deg_+(q)-\deg_-(q)\leq n-m-1,
\]
a contradiction. Thus the images of
\(z^m,\ldots,z^{n-1}\) form a
\(\mathcal D_\phi\)-basis of \(T_\phi/T_\phi s\), proving the claim.

There is a natural isomorphism
\[
T_\phi\otimes_S(S/Ss)\cong T_\phi/T_\phi s.
\]
Therefore,
\[
x_{S/Ss}(\phi)
=
\alpha_\phi
\dim_{\mathcal D_\phi}(T_\phi/T_\phi s)
=
\alpha_\phi(n-m).
\]
Every element in the support of \(a_i z^i\) has \(\phi\)-value
\(i\alpha_\phi\). Hence
\[
\max_{h\in\operatorname{supp}(s)}\phi(h)
=
n\alpha_\phi,
\qquad
\min_{h\in\operatorname{supp}(s)}\phi(h)
=
m\alpha_\phi.
\]
Since \(P(s)\) is the convex hull of \(\operatorname{supp}(s)\), it
follows that
\[
\gr_{P(s)}(\phi)
=
\alpha_\phi(n-m)
=
x_{S/Ss}(\phi).
\]
Finally, \(P(S/Ss)=P(s)\), and additivity on
\(K_0(\mathcal T_S)\) gives
\[
x_{W}(\phi)=\gr_{P(W)}(\phi)
\]
for every \(W\in K_0(\mathcal T_S)\).
\end{proof}
 
The next result allows us to detect modules with single polytope.
\begin{lemma}\label{pd1}
Let $\mathcal{E}$ be a division ring, let $T$ denote a crossed product $\mathcal{E} * \mathbb{Z}^2$, and let $\mathcal{D}$ be the Ore ring of fractions of $T$. Let $M$ be a $T$-module that belongs to $\mathcal{T}_{T}$ with no non-trivial $T$-submodules of finite $\mathcal{E}$-dimension. Then there exists a square matrix $A \in \mathrm{Mat}_{k}(T)$ that is invertible over $\mathcal D$ and such that 
\[
M \cong T^k / T^k A.
\]
\end{lemma}
\begin{proof}
We first show that the projective dimension of $M$ is at most $1$. By \cite[Theorem 5.3]{McCR01}, the global dimension of $T$ is 1 or 2. If the global dimension of $T$ is $1$, there is nothing to prove. Thus, assume that the global dimension of $T$ is $2$.

Let $a,b$ be generators of $\mathbb{Z}^2$. Set $R_a = \mathcal{E}*\langle a\rangle$ and $R_b = \mathcal{E}*\langle b\rangle$. Observe that the global dimensions of both $R_a$ and $R_b$ are equal to $1$.

Consider first the case where $M$ is $R_a$-torsion-free, that is there are no non-zero elements $m \in M$ and $r \in R_a$ such that $rm = 0$. Note that $T = R_a * \langle b\rangle$. Since $M$ is $R_a$-torsion-free,   all finitely generated $R_a$-submodules of $M$ are free. Therefore, by \cite[Theorem~9.11]{McCR01}, we have $\operatorname{pd}_T(M) \le 1$. Similarly, we obtain $\operatorname{pd}_T(M) \le 1$ whenever $M$ is $R_b$-torsion-free.

Define
\[
M_a = \{\, m \in M \mid \text{there exists } 0 \ne r \in R_a \text{ such that } rm = 0 \,\}.
\]
Since $R_a\setminus 0$ is a left Ore subset of $T$, $M_a$ is a $T$-submodule and since $M$ does not contain non-trivial $T$-submodules of finite $\mathcal{E}$-dimension, the module $M_a$ is $R_b$-torsion-free. On the other hand, $M/M_a$ is $R_a$-torsion-free. Thus, both $M_a$ and $M/M_a$ have projective dimension at most $1$. Hence, $M$ itself has projective dimension at most $1$.

Therefore, we can write $M \cong P_1/P_2$ for finitely generated projective $T$-modules $P_1$ and $P_2$. By \cite[Theorem 1]{Moody89}, every projective $T$-module is stably free. Hence, we may assume that $P_1 \cong T^m$ and $P_2 \cong T^k$ for some integers $m,k$. Since $\mathcal{D}$ is a flat $T$-module and $\mathcal{D} \otimes_T M = \{0\}$, it follows that $m = k$.
\end{proof}

We can now show the first principal result of this section.

\begin{theorem} \label{thm: Thurston_norm_modules}
    Let $\mathcal{E}$ be a division ring and $H$ a free abelian group of finite rank. Let $S$ denote the crossed product $\mathcal{E} * H$ and let $\mathcal{D}$ be the Ore ring of fractions of $S$. For every finitely generated $S$-module $N$ such that $\mathcal{D} \otimes_S N$ is trivial, the Thurston map $x_N$ is subadditive, that is
    \[
        x_{N}(\phi_1 + \phi_2) \leq x_N(\phi_1) + x_N(\phi_2) \quad \mbox{for $\phi_1, \phi_2 \in H^1(H; \mathbb{Q})$.}
    \]
\end{theorem}
\begin{proof}
    The statement is trivial if $\phi_1$ and $\phi_2$ are linearly dependent over $\mathbb{Q}$, so assume they are linearly independent. Set
\[
G_0 = \ker \phi_1 \cap \ker \phi_2.
\]
Then $\overline H=H / G_0 \cong \mathbb{Z}^2$. 

Let $\mathcal{D}_0$ be the Ore localization of $\mathcal{E} * G_0$, and let $T$ be the subring of $\mathcal{D}$ generated by $\mathcal{D}_0$ and $S$. Then $T$ is isomorphic to the crossed product $\mathcal{D}_0 * \overline H$,
and $\mathcal{D}$ is the Ore localization of $T$.

Consider the $T$-module $\widetilde{M} = T \otimes_S N$. We have
\[
\mathcal{D} \otimes_T \widetilde{M} \cong \mathcal{D} \otimes_S N = \{0\}.
\]
Therefore, $\widetilde M$   belongs to $\mathcal{T}_T$.
Since $T$ is Noetherian, there exists $M_0$ the maximal $T$-submodule of $\widetilde{M}$ of finite $\mathcal{D}_0$-dimension. Set $M = \widetilde{M} / M_0$. Note that $M$   belongs to $\mathcal{T}_T$ because $\widetilde M$ does. Then, by \cref{pd1}, there exists a square matrix $A \in \operatorname{Mat}_k(T)$ that is invertible over $\mathcal D$ and such that
\[
M \cong T^k / T^k A.
\]

So, by \cref{kielak}, $P(M)$ is a single polytope. 
  
For each $L \in \mathcal{T}_T$ and $\phi \in H^1(\overline H; \mathbb{Q})$, define
\[
x_L(\phi) = \alpha_\phi\cdot \dim_{\mathcal{D}_{\phi}}\!\bigl(T_{\phi} \otimes_T L\bigr).
\]
Since $M_0$ has finite $\mathcal{D}_0$-dimension, it follows that for every $m\in M_0$ there exists $r\in R_\phi\setminus \{0\}$ such that $rm=0$. Therefore,
$$T_{\phi} \otimes_T M_0 \cong  D_{\phi} \otimes_{R_\phi} M_0 =\{0\}.$$  This implies that $x_{\widetilde M}=x_M+x_{M_0}=x_M$, and so by \cref{thick}, $x_{\widetilde M}=\gr_{P(M)}$, which is subadditive being $P(M)$ a single polytope. Therefore
\[
    x_N(\phi_1 + \phi_2) = x_{\widetilde{M}}\left(\overline{\phi_1 + \phi_2}\right) \leq x_{\widetilde{M}}\left(\overline{\phi_1}\right) + x_{\widetilde{M}}\left(\overline{\phi_2}\right) = x_N(\phi_1) + x_N(\phi_2)
\]
where the bar $\bar{}$ stands for the natural maps induced on $H^1(\overline H; \mathbb{Q})$. This shows that $x_N$ is subadditive. 
\end{proof}

As a consequence, we deduce that $x_N$ is continuous on $H^1(H; \mathbb{Q})$, and hence,  $x_N$ extends uniquely by continuity (using rational absolute homogeneity and subadditivity)   to a  seminorm on $H^1(H; \mathbb{R})$, which we call the \emph{Thurston norm of $N$}. We can now prove the second principal result of this section.

\begin{proof}[Proof of \cref{existencepolytope}]
Put
\[
B=\{\phi\in H^1(H;\RR)\mid x_N(\phi)\le 1\} \textrm{\ and\ }
P = B^{*}.
\]
Note that by the definition of the Thurston norm of $N$, the set $P$ is symmetric with respect to the origin.

\begin{claim}
    We have 
    \[
      \gr_{[P]} = \gr_{\,P(N) + \overline{P(N)}}.
    \]
\end{claim}

\begin{proof} We prove the claim assuming that $x_N$ is a norm. In the general case, where $x_N$ is a seminorm, the proof is similar but should be realized in $H^1(H; \mathbb R)/\ker x_N$.

    By definition of \(P\), we have that if $0\ne \phi\in H^1(H; \mathbb R)$, then
    \[
        \gr_{[P]}(\phi)
        = 2 \cdot \frac{1}{\max \{c \in \RR \mid c \phi \in B\}} 
        = 2\, x_N(\phi).
    \]
    On the other hand, taking into account \cref{thick}, we obtain that
    \[
        \gr_{\,P(N) + \overline{P(N)}}(\phi)
        = \gr_{P(N)}(\phi) + \gr_{\overline{P(N)}}(\phi)
        = x_N(\phi) + x_N(\phi)
        = 2\,x_N(\phi),
    \]
    where we used that \(\gr_{\overline{P(N)}} = \gr_{P(N)}\).

    Thus,
    \[
        \gr_{[P]} = \gr_{\,P(N) + \overline{P(N)}}. \qedhere
    \]
\end{proof}

We can write 
\[
P(N) + \overline{P(N)} = [P_1 - P_2]
\]
for some symmetric integral polytopes \(P_1\) and \(P_2\).  
Hence 
\[
\gr_{[P_1]} = \gr _{P(N) + \overline{P(N)}+[P_2]}=\gr_{P(N) + \overline{P(N)}}+\gr_{[P_2]}=\gr_{[P]}+\gr_{[P_2]}= \gr_{[P + P_2]}.
\]
Hence $\gr_{P_1}=\gr_{P + P_2}$.
By \cref{thickfunction}, this implies 
\[
P_1 = P+ P_2.
    \]
 Recall that \(P\) is a closed convex subset. Therefore, to show that \(P\) is a polytope it suffices to prove that it has only finitely many faces, since \(P\) is the convex hull of its vertices.

For any character \(\phi \in H^{1}(H;\mathbb{R})\), we have  
\begin{equation}
   \label{faces} 
F_{\phi}(P_1) = F_{\phi}(P) + F_{\phi}(P_{2}).
\end{equation}

Notice that $F_{\phi}(P)$ is determined uniquely by this equality. Because both \(P_{1}\) and \(P_{2}\) have only finitely many faces, it follows that \(P\) also has finitely many faces, since each face of $P_1$ is determined by the corresponding pair of faces of $P$ and $P_2$.

By \cref{faces}, if $F_{\phi}(P)$  is a vertex, it should have integral coefficientes because $F_{\phi}(P_1)$ and $F_{\phi}(P_2)$ are integral polytopes.  Thus, $P$ is an integral polytope. Finally we have
\[
[P]=[P_1-P_2]= P(N)+\overline{P(N)}.
\]
\end{proof}

\begin{problem}
   Is it true that, under the assumptions of \cref{existencepolytope},   the class $P(N)$ is represented by a single polytope?
\end{problem}

\subsection{Proof of \texorpdfstring{\cref{thurston}}{Theorem \ref*{thurston}} and \texorpdfstring{\cref{higher}}{Theorem \ref*{higher}}}

We prove \cref{higher}, which encompasses \cref{thurston} as a special case.

Since $G$ satisfies the Strong Atiyah Conjecture over $\mathbb Q$, the Linnell ring $\mathcal{D}(\QQ G)$ is a division ring. Set $\mathcal{D}:= \mathcal{D}(\QQ G)$. Let $U$ be the intersection of the kernels of all homomorphisms $G \to \mathbb{Z}$ arising from $H^1(G;\mathbb{Z})$. Denote by $\mathcal{E}$ the division closure of $\QQ U$ in $\mathcal{D}$, and let $S$ be the subring of $\mathcal{D}$ generated by $\mathcal{E}$ and $G$. Observe that $S$ is isomorphic to a crossed product $\mathcal{E} * H$, where $H = G/U$, and that $\mathcal{D}$ is the Ore division ring of fractions of $S$.

Since $G$ is of type $\FP_k(\QQ)$, there exists an exact sequence
\[
0 \to I_k \xrightarrow{\varphi_k} \QQ G^{d_{k-1}} \xrightarrow{\varphi_{k-1}} \cdots \to \QQ G^{d_0} \to \QQ \to 0,
\]
where $I_k$ is a finitely generated $\QQ G$-module. In particular, we obtain a short exact sequence
\[
0 \to I_k \xrightarrow{\varphi_k} \QQ G^{d_{k-1}} \xrightarrow{\varphi_{k-1}} \operatorname{Im} \varphi_{k-1} \to 0.
\]

Consider the module
\[
N = H_k(G;S) = \operatorname{Tor}_k^{\QQ G}(S,\QQ) \cong \operatorname{Tor}_1^{\QQ G}(S, \operatorname{Im} \varphi_{k-1}).
\]
It follows that $N$ is an $S$-submodule of $S \otimes_{\QQ G} I_k$. Since $S \otimes_{\QQ G} I_k$ is a finitely generated $S$-module and $S$ is Noetherian, we conclude that $N$ is finitely generated. By flatness of the Ore localization, $$\dim_{\mathcal D} \mathcal D\otimes_S N=\dim_{\mathcal D} H_k(G;\mathcal D)=b_k^{(2)}(G)=0.$$ Since $\mathcal D$ is a division ring $\mathcal D\otimes_S N=0$. 

As in the previous subsection, for each non-trivial character $\phi \colon H \to \mathbb{Z}$, define
\[
R_{\phi} := \mathcal{E} * \ker(\phi),
\]
and let $\mathcal{D}_{\phi}$ denote its Ore localization. Let $T_{\phi}$ be the subring of $\mathcal{D}$ generated by $\mathcal{D}_{\phi}$ and $S$. Then $T_{\phi}$ is a crossed product
\[
T_{\phi} \cong \mathcal{D}_{\phi} * \bigl(H / \ker \phi\bigr).
\]

Observe also that $T_\phi$ is an Ore localization of $S$: indeed, it suffices to invert the nonzero elements of $R_\phi$. Consequently, $T_\phi$ is flat as an $S$-module.

Let $\widetilde{\phi}$ be the map obtained as a composition of the projection $G \to H$ and $\phi \colon H \to \ZZ$. Note that as $\QQ \ker \widetilde{\phi}$-rings the division rings
\[
    \mathcal{D}_{\phi} \quad \mbox{and} \quad \mathcal{D}\left(\QQ \ker  {\phi}\right)
\]
are isomorphic, and note also that as right $\QQ G$-modules the projection $G \to H$ induces the following isomorphism
\[
    \mathcal{D}_{\phi} \otimes_{\QQ \ker \widetilde{\phi}} \QQ G \longrightarrow \mathcal{D}_{\phi} \otimes_{\QQ \ker \phi} \QQ H \cong T_{\phi}.
\]

Define a function on $H^1(H;\mathbb{Q})$ by assigning to each $\phi \in H^1(H;\mathbb{Q})$ the value
\[
x_N(\phi) := \alpha_\phi \cdot \dim_{\mathcal{D}_{\phi}}\!\bigl(T_{\phi} \otimes_S N\bigr).
\]

Since $T_\phi$ is flat as an $S$-module, we obtain
\[
\dim_{\mathcal{D}_{\phi}}\!\bigl(T_{\phi} \otimes_S N\bigr)
= \dim_{\mathcal{D}_{\phi}} \left( \operatorname{Tor}_k^{\QQ G}(T_\phi, \QQ) \right)
= \dim_{\mathcal{D}_{\phi}} \left( \operatorname{Tor}_k^{\QQ \ker \widetilde{\phi}}(\mathcal{D}_\phi, \QQ) \right)
= b_k^{(2)}\left(\ker \widetilde{\phi}\right).
\]
Notice that in the second equality we use Shapiro's lemma.
By \cref{thick}, we have
\[
2x_N(\phi) = 2\gr_{P(N)}(\phi) 
= \gr_{ P(N)+\overline{P(N)} )}(\phi).
\]

Finally, by \cref{existencepolytope},   $  P(N)+\overline{P(N)} $ is represented by a single symmetric integral polytope. This completes the proof of \cref{higher}.

\section{A Criterion for Cohomological Goodness}\label{sec:good}

A group $G$ is called \emph{$p$-good} for a given prime number $p$ if the homomorphism of cohomology groups
\[
    H^n(\widehat{G}, M) \to H^n(G, M)
\]
induced by the natural homomorphism $G \to \widehat{G}$ of $G$ to its profinite completion $\widehat{G}$ is an isomorphism for every finite $p$-primary $G$-module $M$. The group $G$ is called \emph{good} if it is $p$-good for every prime $p$. This important concept was introduced by J-P. Serre in \cite[Section I.2.6]{Serre1997}.

In this section we will provide a criterion for $p$-goodness that is inspired in \cite[Theorem 5.6]{GrunewaldJaikinPintoZalesskii2014}.

\begin{proposition}\label{prop: p-goodness}
   Let $G$ be a group of type $\mathrm{FP}_2(\mathbb{Z})$ with $\mathrm{cd}_{\mathbb{Z}}(G) \leq 2$ and $b_2^{(2)}(G) = 0$, and let $p$ be a prime number. If $G$ is virtually residually $p$-finite, then $G$ is $p$-good.
\end{proposition}
\begin{proof} Since $p$-goodness is a commensurability invariant, we can replace $G$ with a subgroup of finite index and assume that $G$ is residually $p$-finite.

Since $\operatorname{cd}_{\mathbb{Z}}(G) \le 2$, following \cite[Section~I.2.6]{Serre1997}, we know that in order to show that $G$ is $p$-good, it is enough to show that for any subgroup $H$ of finite index, the map
\begin{equation}
\label{pgoodcondition}
    H^2(\widehat{H}, \mathbb{F}_p) \to H^2(H, \mathbb{F}_p)
\end{equation}
is bijective. In the same section, Serre also proves that this map is injective. Thus, it suffices to show the equality of the dimensions of the involved $\mathbb{F}_p$-vector spaces.

   Since $H$ is of cohomological dimension at most $2$ and of type $\operatorname{FP}_2(\mathbb{Z})$, there exists the following resolution of the trivial $\mathbb{Z}H$-module $\mathbb{Z}$:
\[ 0 \to P \to \mathbb{Z}H^d \to \mathbb{Z}H \to \mathbb{Z} \to 0, \]
where $P$ is a finitely generated projective $\mathbb{Z}H$-module. Let $k$ be such that $\mathbb{Z} \otimes_{\mathbb{Z}H} P \cong \mathbb{Z}^k$. Then we also have $\mathbb{F}_p \otimes_{\mathbb{Z}H} P \cong \mathbb{F}_p^k$.

By \cite[Theorem~8]{Eckmann_StrongBass}, we have $\beta_0^{\mathbb{Q}H}(\mathbb{Q}H \otimes_{\mathbb{Z}H} P) = k$. In particular, we obtain:
\begin{multline*}
\dim H_1(H, \mathbb{F}_p) - \dim H_2(H, \mathbb{F}_p) - 1 = d - \dim_{\mathbb{F}_p} (\mathbb{F}_p \otimes_{\mathbb{Z}H} P) - 1 \\
= d - \beta_0^{\mathbb{Q}H}(\mathbb{Q}H \otimes_{\mathbb{Z}H} P) - 1 = b_1^{(2)}(H) - b_2^{(2)}(H) - b_0^{(2)}(H) = b_1^{(2)}(H).
\end{multline*}

Let $H_{\hat{p}}$ be the pro-$p$ completion of $H$. Since $H$ is residually $p$-finite, we can see $H$ as a subgroup of $H_{\hat p}$.
In order to show that \cref{pgoodcondition} is surjective, it is enough to show that 
\begin{equation}
\label{propgoodcondition}
    H^2(H_{\hat{p}}, \mathbb{F}_p) \to H^2(H, \mathbb{F}_p)
\end{equation}
is surjective. Since  the inflation map $H^2(H_{\hat{p}}, \mathbb{F}_p) \to H^2(\widehat H, \mathbb{F}_p)$ is always injective, the inflation map $H^2(H_{\hat{p}}, \mathbb{F}_p) \to H^2(  H, \mathbb{F}_p)$ is also injective.

Suppose for the sake of contradiction that $\dim_{\mathbb{F}_p} H^2(H_{\hat{p}}, \mathbb{F}_p) < \dim_{\mathbb{F}_p} H^2(H, \mathbb{F}_p)$. Since $H^1(H_{\hat{p}}, \mathbb{F}_p) \cong H^1(H, \mathbb{F}_p)$, it follows that:
\[
\dim_{\mathbb{F}_p} H^1(H_{\hat{p}}, \mathbb{F}_p) - \dim_{\mathbb{F}_p} H^2(H_{\hat{p}}, \mathbb{F}_p) > \dim_{\mathbb{F}_p} H^1(H, \mathbb{F}_p) - \dim_{\mathbb{F}_p} H^2(H, \mathbb{F}_p) = b_1^{(2)}(H) + 1.
\]
Define $\delta(H_{\hat{p}}) = \dim_{\mathbb{F}_p} H^1(H_{\hat{p}}, \mathbb{F}_p) - \dim_{\mathbb{F}_p} H^2(H_{\hat{p}}, \mathbb{F}_p)$. Recall that $\dim_{\mathbb{F}_p} H^1(H_{\hat{p}}, \mathbb{F}_p)$ is the minimal number of generators of $H_{\hat{p}}$, and $\dim_{\mathbb{F}_p} H^2(H_{\hat{p}}, \mathbb{F}_p)$ is the minimal number of relations of $H_{\hat{p}}$. Using the pro-$p$ version of the Schreier formula, we obtain that for every open subgroup $U$ of $H_{\hat{p}}$:
\[
\dim_{\mathbb{F}_p} H^1(U, \mathbb{F}_p) - \dim_{\mathbb{F}_p} H^2(U, \mathbb{F}_p) - 1 \geq  (\delta(H_{\hat{p}}) - 1) [H_{\hat{p}} : U].
\]
In particular, we have:
\[
\dim_{\mathbb{Q}_p} (\mathbb{Q}_p \otimes_{\mathbb{Z}_p} U/[U,U]) \geq (\delta(H_{\hat{p}}) - 1) [H_{\hat{p}} : U].
\]
Since $U \cap H$ is dense in $U$, it follows that:
\[
\dim_{\mathbb{Q}} H_1(U \cap H, \mathbb{Q}) \geq (\delta(H_{\hat{p}}) - 1) [H : U \cap H].
\]
Fix a chain of open normal subgroups $U_i$ of $H_{\hat{p}}$ such that $U_{i+1} \leq U_i$ and $\bigcap U_i = \{1\}$. Let $H_i = U_i \cap H$. Since $H$ is of type $\operatorname{FP}_2(\mathbb{Z})$, we can apply Lück's Approximation Theorem  with respect to the chain $\{H_i\}$ to obtain:
\[
b_1^{(2)}(H) = \lim_{i \to \infty} \frac{\dim_{\mathbb{Q}} H_1(H_i, \mathbb{Q})}{[H : H_i]} \geq \delta(H_{\hat{p}}) - 1 > b_1^{(2)}(H).
\]
This is a contradiction, which completes the proof.
\end{proof}

\begin{proposition}\label{prop: goodness_FbyZ}
    Finitely generated free-by-cyclic groups are good.
\end{proposition} 
\begin{proof}
    Let $G$ be a finitely generated free-by-cyclic group. Then it is of cohomological dimension is at most 2. By \cite[Theorem 1.2]{Linton_FbyZembedding}, $G$ embeds into a (finitely generated free)-by-cyclic group $G'$. Let $p$ be a prime number, according to \cite[Theorem 1.1]{JankiewiczSchreve_vRp} $G'$, and hence $G$, is virtually residually $p$-finite. Since $G$ is finitely presented by \cite{FeighnHandel1999}, $G$ is $p$-good by \cref{prop: p-goodness}. Therefore, $G$ is good.
\end{proof}

\section{Torsion-free virtually free-by-cyclic groups are Lewin}\label{sec:Lewin}

In this section we show that a torsion-free virtually free-by-cyclic group $G$ is a Lewin group. Furthermore, given a division ring $E$ and a crossed product $E * G$, we prove that $E * G$ is a pseudo-Sylvester domain. To do so, we follow the strategy of the third author in \cite{SanchezPeralta_Knots}, that is, first we will construct a Hughes-free division ring embedding and then use the following criterion of Henneke and L\'opez-\'Alvarez.

\begin{theorem}[{\cite[Theorem 3.4]{HL22}}] \label{thm: pseudoSyl_criterion}
    Let $\mathcal{D}$ be a division $R$-ring of fractions. Assume that
    \begin{enumerate}
        \item $\Tor_1^R(\mathcal{D}, \mathcal{D}) = 0$ and
        \item for any finitely generated left or right $R$-submodule $M$ of $\mathcal{D}$ and any exact sequence $0 \to J \to R^n \to M \to 0$, the $R$-module $J$ is finitely generated stably free.
    \end{enumerate}
    Then $R$ is a pseudo-Sylvester domain and $\mathcal{D}$ is the universal division $R$-ring of fractions.
\end{theorem}

We recall that an $R$-module $M$ is called \emph{stably free} if there exists a non-negative integer $n$ such that $M \oplus R^n$ is a free $R$-module.

Note that if $G$ is virtually free, then by \cite[Theorem B]{Swan1969} $G$ is actually a free group. In this case, $E * G$ is known to be a free ideal ring \cite{Cohn_Universality}. Consequently, the results listed below--including \cref{thm: univ_vFbyZ}--are already established and admit more straightforward proofs. For the remainder of the section, let $H$ denote a finite-index normal subgroup of $G$ that is free-by-cyclic (with the free-by-$\Z$ case being of primary interest).  

First of all, we start showing that $G$ is indeed a locally indicable group. The following is a generalization of \cite[Theorem 3.11]{JL23}.

\begin{proposition}\label{prop: vanishing_b2_li}
    Let $G$ be a torsion-free group of cohomological dimension at most $2$ over $\QQ$ satisfying the Weak Atiyah Conjecture. If $b_2^{(2)}(G)$ vanishes, then $G$ is locally indicable. 
\end{proposition}
\begin{proof}
    Let $H$ be a finitely generated non-trivial subgroup of $G$. Consider the exact sequence
    \[
        0 \to P \to \QQ H^d \to I_{\QQ H} \to 0
    \]
    where $d$ is a positive integer and $P$ a projective $\QQ H$-module since  $H$ is of cohomological dimension at most $2$ over $\QQ$. By \cite[Theorem 3.9]{JL23} $b_2^{(2)}(H)$ vanishes, and hence
    \[
        d = \beta_0^H(\QQ H^d) = \beta_0^H(P) + \beta_0^H(I_{\QQ H}).
    \]
    Since $\mathcal{R}_{\QQ H}$ is a semisimple ring (observe that $H$ also satisfies the Weak Atiyah Conjecture) and $\dim_{\mathcal{R}_{\QQ H}}$ is discrete and faithful over projective modules, $\mathcal{R}_{\QQ H} \otimes_{\QQ H} P$ must be a finitely generated $\mathcal{R}_{\QQ H}$-module. Therefore, $P$ is a finitely generated $\QQ H$-module according to \cite[Lemma 4]{LinnellLuckSchick_Ore}. By \cite[Theorem 8]{Eckmann_StrongBass}, it holds that
    \[
        \beta_0^{\QQ}(P)=\dim_{\QQ}(\QQ \otimes_{\QQ H} P) = \beta_0^H(P).
    \]
    Thus,
    \[
        b_1(H) = d - \beta_0^{\QQ}(P) + b_2(H) = \beta_0^H(I_{\QQ H}) + b_2(H) \geq 1.
    \]
    Hence $H$ is indicable, and so $G$ is locally indicable.
\end{proof}

\begin{lemma} \label{lem: vFbyZ_li}
	Let $G$ be a torsion-free virtually free-by-cyclic group. Then $G$ is a locally indicable group of cohomological dimension at most $2$.
\end{lemma}
\begin{proof}
	Since $G$ is torsion-free and $H$ is of finite index in $G$, according to \cite[Th\'eor\`eme 1, Section 1.7]{Serre71} $\cd_{\Z}(G) = \cd_{\Z}(H) \le 2$. Moreover, $H$ satisfies the Strong Atiyah Conjecture over $\CC$ by \cite{Linnell1993} and $b_2^{(2)}(H)$ vanishes (see, for instance, \cite[Th\'eor\`eme 6.6]{Gaboriau2002}). Hence, $G$ satisfies the Weak Atiyah Conjecture and $b_2^{(2)}(G) = 0$. Therefore, by \cref{prop: vanishing_b2_li}, $G$ is locally indicable.
\end{proof}

The existence of the Hughes-free division ring will follow as in \cite[Theorem 3.4]{SanchezPeralta_Knots}.

\begin{theorem} \label{thm: HF_vFbyZ}
	Let $G$ be a torsion-free virtually free-by-cyclic group. Then $G$ is Hughes-free embeddable. Moreover, if $H$ denotes a finite-index normal subgroup of $G$ that is free-by-cyclic, then for any crossed product $E * G$, where $E$ is a division ring, $\mathcal{D}_{E * G} = \mathcal{D}_{E * H} * G/H$.
\end{theorem}
\begin{proof}
    The proof follows verbatim from that of \cite[Theorem 3.4]{SanchezPeralta_Knots}. More concretely, one only needs to replace goodness of fundamental groups of compact $3$-manifold groups in the proof of \cite[Theorem 3.4]{SanchezPeralta_Knots} by \cref{prop: goodness_FbyZ}, which ensures that all finitely generated free-by-cyclic groups are good (in the sense of Serre). Therefore the later application of \cite[Theorem 4.2]{FisherSanchezPeralta_KZDC} is justified in this broader context.
\end{proof}

Next, we examine the properties of the Hughes-free embedding. These properties not only help establish the universality of the Hughes-free division ring, but they are also of independent interest and play a central role in \cref{sec: l2_rigid}. Recall that over a ring $R$, an $R$-module $M$ is \emph{coherent} if every finitely generated $R$-submodule $N$ of $M$ is a finitely presented $R$-module.

\begin{proposition} \label{prop: properties_D}
    Let $E$ be a division ring, $G$ a torsion-free virtually free-by-cyclic group and $E * G$ a crossed product. Then the Hughes-free division $E * G$-ring $\mathcal{D}_{E * G}$ satisfies the following properties:
    \begin{enumerate}
        \item $\mathcal{D}_{E * G}$ is coherent.
        \item $\Tor_1^{E * G}(\mathcal{D}_{E * G}, \mathcal{D}_{E * G}) = \{0\}$.
        \item The $E*G$-module $\mathcal{D}_{E * G}$ is of weak dimension at most $1$.
        \item Every finitely generated $E * G$-submodule of $\mathcal{D}_{E * G}$ is of projective dimension at most $1$.
    \end{enumerate}
\end{proposition}
\begin{proof}
    The proof is analogous to that of \cite[Propositions 3.5 \& 3.6]{SanchezPeralta_Knots} using \cref{thm: HF_vFbyZ}.
\end{proof}

We will further need the following result on the structure of finitely generated projective $E * G$-modules.

\begin{theorem}\label{thm: FJ_vFbyZ}
    Let $E$ be a division ring, $G$ a torsion-free virtually free-by-cyclic group and $E * G$ a crossed product. Then the embedding $E \hookrightarrow E * G$ induces an isomorphism
    \[
        K_0(E) \xrightarrow{\cong} K_0(E \ast G).
    \]
    In particular, every finitely generated projective $E * G$-module is stably free.
\end{theorem}
\begin{proof}
    By \cite[Theorem 1.1]{BestvinaFujiwaraWigglesworth_FJHbyZ} $H$ satisfies the $K$-theoretic Farrell--Jones Conjecture with coefficients in an additive category. Now, since the conjecture is preserved under finite-index   overgroups \cite[Section 2.3]{Wegner_FJInheritance}, the statement follows from \cite[Proposition 4.2]{HL22}.
\end{proof}

We conclude the section showing that $G$ is a Lewin group.

\begin{proof}[Proof of \cref{thm: univ_vFbyZ}]
    The same argument in the proof of \cite[Theorem 1.5]{SanchezPeralta_Knots} replacing \cite[Propositions 3.5 \& 3.6]{SanchezPeralta_Knots} by \cref{prop: properties_D}, and \cite[Theorem 3.7]{SanchezPeralta_Knots} by \cref{thm: FJ_vFbyZ}, shows that the Hughes-free division $E * G$-ring $\mathcal{D}_{E * G}$ given by \cref{thm: HF_vFbyZ} satisfies the conditions of \cref{thm: pseudoSyl_criterion}. Therefore, $E * G$ is a pseudo-Sylvester domain and $\mathcal{D}_{E * G}$ is also the universal division $E * G$-ring of fractions. In particular, $G$ is a Lewin group.
\end{proof}

\section{\texorpdfstring{$L^2$}{L2}-rigidity of torsion-free virtually free-by-cyclic groups} \label{sec: l2_rigid}

In this section we will show \cref{compressedL2}. We follow the strategy of \cite{Ja24} and examine the $\mathcal{D}$-torsion part of $k G$-modules over a given division $k G$-ring $\mathcal{D}$, where $k$ will stand for a field throughout the section. 

\subsection{The \texorpdfstring{$\UO$}{U}-torsion part}
    
    Let $\UO$ be a $k G$-ring and let $M$ be a $k G$-module. We are interested in understanding the kernel of the natural map
    \[
        M \to\mathcal U \otimes_{k G} M, \quad x \mapsto 1 \otimes x.
    \]
    We call this kernel the \emph{$\UO$-torsion part of $M$} and denote it by $\tor_{k G}(M;\mathcal U)$ or simply $\tor_{k G}(M)$ if the $k G$-ring $\UO$ is clear from the context. We say that $M$ is \emph{$\UO$-torsion-free} if
    \[
        \tor_{k G}(M) = \{0\}.
    \]
    Recall that for a given subgroup $H$ in $G$, we denote by $I_k(G,H)$ the $k G$-module $I_{k G}/^G I_{k H}$. We will be especially interested in the $\UO$-torsion part of $I_k(G,H)$. We say that $H$ is \emph{$\UO$-closable} in $G$ if there exists a subgroup $\widetilde{H}$ of $G$ such that 
    \[
        \tor_{k G}(I_k(G,H)) = {}^G I_k(\widetilde{H}, H).
    \]
    In that case, $\widetilde{H}$ is unique and contains $H$ (see \cite[Lemma 2.1]{Ja24}), and we call the subgroup $\widetilde{H}$ the \emph{$\UO$-closure} of $H$ in $G$. We say that $G$ is \emph{$\UO$-subgroup rigid} if every finitely generated subgroup is $\UO$-closable in $G$.

    Whenever we consider the $\QQ[G]$-ring $\UO(G)$, we will speak about $L^2$-closable, closure or subgroup rigidity instead of using the ring $\UO(G)$.

    Let $\mathcal{D}$ be a division $k G$-ring. We record how the properties of $\mathcal{D}$ reflect on the $\mathcal{D}$-torsion part. First, we consider Linnell-free division ring embeddings. 

    \begin{lemma} \label{extabs}  
        Let $\mathcal{D}$ be a Linnell-free division $kG$-ring and $H$ a subgroup of $G$. Denote by $\mathcal{D}_H$ the division closure of $kH$ in $\mathcal{D}$, and let $M$ be a $\mathcal D_{H}$-torsion-free $kH$-module. Then the $kG$-module $kG\otimes_{kH} M$ is $\mathcal D$-torsion-free.
    \end{lemma}
    \begin{proof}
        The argument is identical to the proof of \cite[Proposition 3.3]{Ja24}.
    \end{proof}

    \begin{lemma} \label{lem: ind_tors} 
        Let $\mathcal{D}$ be a Linnell-free division $kG$-ring and $H \leq U \leq L \leq G$ a chain of groups. Then the $kL$-module 
        \[
            \tor_{kL}\left( ^L I_k(U,H); \mathcal{D}\right)
        \]
        is equal to the $kL$-submodule of $^L I_k(U,H)$ generated by 
        \[
            \tor_{kU}\left(I_k(U,H); \mathcal{D}\right).
        \]
    \end{lemma}
    \begin{proof}
        The proof is analogous to that of \cite[Claim 3.7]{Ja24} using \cref{extabs}.
    \end{proof}

    Next, we consider homological properties of $\mathcal{D}$ as a $k G$-module.

    \begin{lemma} \label{lem: tf_vanishing_b1}
        Let $\mathcal{D}$ be a division $kG$-ring. Assume that $\mathcal{D}$ is of weak dimension at most $1$ as right $kG$-module and that $\Tor_1^{kG}(\mathcal{D}, \mathcal{D})$ vanishes. Then for every $\mathcal{D}$-torsion-free $kG$-module $M$ it holds that
        \[
            \beta_1^{\mathcal{D}}(M) = 0.
        \]
    \end{lemma}
    \begin{proof}
        By assumption $M$ is a $kG$-submodule of $\mathcal{D} \otimes_{kG} M$. Thus, by \cref{lem: monotone_b1}, it suffices to show that $\beta_1^{kG}(\mathcal{D} \otimes_{kG} M) = 0$. Now, $\mathcal{D}$ is a division ring, and hence
        \[
            \mathcal{D} \otimes_{kG} M \cong \mathcal{D}^{\kappa}
        \]
        for some cardinal    $\kappa$. Therefore, we get
        \[
            \Tor_1^{kG}(\mathcal{D}, \mathcal{D} \otimes_{kG} M) \cong \bigoplus_{\kappa} \Tor_1^{kG}(\mathcal{D}, \mathcal{D}) = \{0\},
        \]
        which concludes the proof.
    \end{proof}

    \begin{lemma} \label{lem: properties_tors}
        Let $\mathcal{D}$ be a division $kG$-ring. Assume that $\mathcal{D}$ is of weak dimension at most $1$ as right $kG$-module and that $\Tor_1^{kG}(\mathcal{D}, \mathcal{D})$ vanishes. Then for every $kG$-module $M$ it holds that
        \[
            \beta_0^{\mathcal{D}}(\tor_{kG}(M)) = 0.
        \]
        Moreover, $\tor_{kG}(M)$ is the maximal $kG$-submodule of $M$ with this property and
        \[
            \beta_1^{\mathcal{D}}(\tor_{kG}(M)) = \beta_1^{\mathcal{D}}(M).
        \]
    \end{lemma}

    \begin{proof}
        Denote by $\overline{M}$ the image of the natural map $M \to \mathcal{D} \otimes_{kG} M$. Observe that $\mathcal{D} \otimes_{kG} \overline{M} \cong \mathcal{D} \otimes_{kG} M$, and so, $\overline{M}$ is $\mathcal{D}$-torsion-free and $\beta_0^{\mathcal{D}}(\overline{M}) = \beta_0^{\mathcal{D}}(M)$. Then, by \cref{lem: tf_vanishing_b1}, $\beta_1^{\mathcal{D}}(\overline{M})$ vanishes. Thus, the statements follow from extending scalars from $kG$ to $\mathcal{D}$ on the exact sequence
        \[
            0 \to \tor_{kG}(M) \to M \to \overline{M} \to 0
        \]
        and taking dimensions by \cref{lem: monotone_b1}.
    \end{proof}

    We conclude the section looking at the coherence property.

    \begin{lemma} \label{lem: coh_tf_fg}
        Let $\mathcal{D}$ be a coherent division $kG$-ring and $M$ a finitely generated $kG$-module. Let $J$ be a $kG$-submodule of $M$ such that $M/J$ is $\mathcal D$-torsion-free. Then $J$ is finitely generated. 
    \end{lemma}
    \begin{proof}
        Note that $M/J$ is finitely generated, and thus, finitely presented by coherence of $\mathcal{D}$ since $\mathcal{D} \otimes_{k G} M/J \cong \mathcal{D}^n$ for some   $n\ge 0$, because $\mathcal{D}$ a division ring. But then, \cite[Proposition 1.4]{Bieri1981} implies that $J$ is finitely generated.
    \end{proof}

    \begin{corollary}\label{cor: coh_tors_fg}
        Let $\mathcal{D}$ be a coherent division $kG$-ring and $M$ a finitely generated $kG$-module. Then $\tor_{kG}(M)$ is finitely generated. In particular, if $G$ is a finitely generated group and $\widetilde H$ is the $\mathcal D$-closure of a subgroup $H$ of $G$, then $\widetilde H$ is finitely generated. 
    \end{corollary}

\subsection{\texorpdfstring{$L^2$}{L2}-compressed subgroups are \texorpdfstring{$L^2$}{L2}-independent} \label{subsec: F2_rigid}

In this section we show the next result on $L^2$-subgroup rigidity. Then, the proof of \cref{compressedL2} will follow.

\begin{theorem}\label{thm:L2rigid}
    Finitely generated torsion-free virtually free-by-cyclic groups are $L^2$-subgroup rigid.
\end{theorem}

Firstly, we show that $G$ is $\mathcal{D}_{\mathbb{F}_2G}$-rigid, and, secondly, we deduce that $G$ is $\mathcal{D}_{\QQ G}$-rigid. Before going into the proofs, note that by \cref{thm: univ_vFbyZ} for every field $k$ there exists the Hughes-free universal division $k G$-ring of fractions $\mathcal{D}_{k G}$, it is a universal localization of $k G$, and moreover, by \cref{prop: properties_D} it is coherent and of weak dimension at most $1$ as $k G$-module. Thus, the $\mathcal{D}_{k G}$-torsion part of a $k G$-module $M$, which we will denote simply by $\tor_{k G}(M)$, enjoys all the properties we have studied in the previous section. Note also that by Hughes-freeness (see \cref{sec: div_ring_embed}), for any subgroup $H$ of $G$ and any $k H$-module $M$, $\tor_{k H}(M)$ identifies with $\tor_{k H}(M; \mathcal{D}_{k G})$.

\begin{theorem}\label{thm:DF2Grigid}
    A finitely generated torsion-free virtually free-by-cyclic group $G$ is  $\mathcal{D}_{\mathbb{F}_2G}$-subgroup rigid.
\end{theorem}
\begin{proof}
    Let $H$ be a subgroup of $G$. We define $\mathcal{S}$ as the set of all subgroups $L$ of $G$ containing $H$ and such that
    \[
        \tor_{\mathbb{F}_2 G}\left(I_{\FF_2}(G,H)\right) \subseteq {}^G I_{\FF_2}(L,H).
    \]
    For every $L \in \mathcal{S}$ it holds that
    \[
        \tor_{\mathbb{F}_2 G}\left(I_{\FF_2}(G,H)\right) = {}^G \tor_{\mathbb{F}_2 L}\left(I_{\FF_2}(L,H)\right).
    \]
    Indeed, by \cref{lem: ind_tors}
    \[
        \tor_{\mathbb{F}_2 G}\left({}^GI_{\FF_2}(L,H)\right) = {}^G \tor_{\mathbb{F}_2 L}\left(I_{\FF_2}(L,H)\right).
    \]
    Thus, by \cref{lem: properties_tors}, it is the largest $\FF_2 G$-submodule of ${}^G I_{\FF_2}(L,H)$ with vanishing dimension, so
    \[
        \tor_{\mathbb{F}_2 G}\left(I_{\FF_2}(G,H)\right) \subseteq {}^G \tor_{\mathbb{F}_2 L}\left(I_{\FF_2}(L,H)\right).
    \]
    But, again by \cref{lem: properties_tors}
    \[
        {}^G \tor_{\mathbb{F}_2 L}\left(I_{\FF_2}(L,H)\right) \subseteq \tor_{\mathbb{F}_2 G}\left(I_{\FF_2}(G,H)\right),
    \]
    and thus, the equality holds.
    \begin{claim}
        The set $\mathcal{S}$ has a minimal element. Moreover, every minimal element of $\mathcal{S}$ is finitely generated.    
    \end{claim}
    \begin{proof}
        Let $\{L_i\}_{i \in I} \subseteq \mathcal{S}$ be a linearly ordered subset and set $L = \bigcap_{i \in I} L_i$. Then, $H \leq L$ and, by \cref{lem: inter_aug_ideals}, $\cap_{i \in I} {}^{G} I_{\mathbb{F}_2 L_i} = {}^{G} I_{\mathbb{F}_2 L}$. In particular,
        \[
            \tor_{\mathbb{F}_2 G} \left(I_{\FF_2}(G,H)\right) \subseteq {}^G I_{\FF_2}(L,H),
        \]
        and so, $L \in \mathcal{S}$. Therefore, by Zorn's Lemma, $\mathcal{S}$ has a minimal element. 
        
        Now, let $L$ be a minimal element of $S$ and let $J$ be the $\FF_2 G$-submodule of $I_{\FF_2 G}$ such that
        \[
            \tor_{\mathbb{F}_2 G} \left(I_{\FF_2}(G,H)\right) = J / {}^G I_{\mathbb{F}_2H}.
        \]
        By \cref{lem: coh_tf_fg}, the $\FF_2 G$-module $J$ is finitely generated. Hence, there exists a finitely generated subgroup $L_0 \leq L$ such that 
        \[
            J \subseteq {}^{G} I_{\mathbb{F}_2 L_0}.
        \]
        By \cite[Lemma 2.1]{Ja24}, $H$ is a subgroup of $L_0$, and then the minimality of $L$ implies that $L = L_0$, showing that it is finitely generated.
    \end{proof}
    Let $L$ be a minimal element in $\mathcal S$. 
    \begin{claim} \label{subgroupindex2} 
        Assume that $\tor_{\mathbb F_2L}(I_{\FF_2}(L,H))\ne I_{\FF_2}(L,H)$. Then there exists a subgroup $U$ of index $2$ in $L$ such that 
        \[
            \tor_{\mathbb F_2L}(I_{\FF_2}(L,H))\subseteq {}^L I_{\FF_2}(U,H).
        \]
    \end{claim}
    \begin{proof}
        The proof is analogous to that of \cite[Claim 3.9]{Ja24} replacing \cite[Lemma 3.6]{Ja24} by \cref{lem: index2_aug_ideal} and using the universality of the division ring showed in \cref{thm: univ_vFbyZ}.
    \end{proof}
    Suppose for a contradiction that $\tor_{\mathbb{F}_2 L}(I_{\FF_2}(L,H)) \ne I_{\FF_2}(L,H)$ and let $U$ be the subgroup obtained in the previous claim. Then, by \cref{lem: ind_tors,lem: properties_tors}, we have
    \[
        \tor_{\mathbb{F}_2 G}\!\left(I_{\FF_2}(G,H)\right) = {}^G \tor_{\mathbb{F}_2L}(I_{\FF_2}(L,H)) = {}^G \tor_{\mathbb{F}_2L}\left({}^L I_{\FF_2}(U,H) \right) \subseteq {}^G I_{\FF_2}(U,H),
    \]
    which contradicts the minimality of $L$. Hence, $\tor_{\mathbb{F}_2 L}(I_{\FF_2}(L,H)) = I_{\FF_2}(L,H)$, and consequently,
    \[
        \tor_{\mathbb{F}_2 G}\!\left(I_{\FF_2}(G,H)\right) = {}^G I_{\FF_2}(L,H).
    \]
\end{proof}

To pass from the $\mathcal{D}_{\FF_2 G}$-closure to the $L^2$-closure we will need to compare the dimension of modules over $\Z G$.

\begin{proposition} \label{prop: dim_QG_F2G}
    Let $G$ be a torsion-free virtually free-by-cyclic group and let $M$ be a finitely presented $\Z G$-module. Then it holds that 
    \[
        \beta_0^G(\QQ \otimes_{\Z} M) \leq \beta_0^{\mathcal{D}_{\FF_2 G}}(\FF_2 \otimes_{\Z} M).
    \]
\end{proposition}
\begin{proof}
    Let $H$ be the finite-index normal subgroup of $G$ that is free-by-cyclic. By \cref{thm: HF_vFbyZ}, $\mathcal{D}_{\QQ G} = \mathcal{D}_{\QQ H} * G/H$. In particular, $\dim_{\mathcal{D}_{\QQ G}} = \frac{1}{\lvert G : H \rvert} \dim_{\mathcal{D}_{\QQ H}}$, and so
    \begin{align*}
        \beta_0^G(\QQ \otimes_{\Z} M) & = \frac{1}{\lvert G : H \rvert} \dim_{\mathcal{D}_{\QQ H}} ( (\mathcal{D}_{\QQ H} * G/H) \otimes_{\QQ G} (\QQ \otimes_{\Z} M)) \\
        & = \frac{1}{\lvert G : H \rvert} \dim_{\mathcal{D}_{\QQ H}} ( \mathcal{D}_{\QQ H} \otimes_{\QQ H} (\QQ \otimes_{\Z} M)) = \frac{1}{\lvert G : H \rvert} \beta_0^H(\QQ \otimes_{\Z} M).
    \end{align*}
    Similarly, $\beta_0^{\mathcal{D}_{\FF_2 G}}(\FF_2 \otimes_{\Z} M) = \frac{1}{\lvert G : H \rvert} \beta_0^{\mathcal{D}_{\FF_2 H}}(\FF_2 \otimes_{\Z} M)$. Note that $M$ is itself a finitely presented $\ZZ H$-module. Thus, it suffices to prove the result whenever $G$ is a free-by-cyclic group. Denote the free normal subgroup by $N$. By \cite[Proposition 2.5(1)]{Jaikin_UniversalityHF}, $\rk_G$ and $\rk_{\FF_2 G}$ are the natural extension of $\rk_N$ and $\rk_{\FF_2 N}$, respectively. Now, since $\Z N$ is a pseudo-Sylvester domain \cite{DicksSontag_Syl} and $\mathcal{D}_{\QQ[N]}$ is the universal division $\Z N$-ring, for every matrix $A'$ over $\Z N$ we have that
    \[
        \rk_N(A') \geq \rk_{\FF_2 N}(A').
    \]
    Thus, if $A$ is a matrix over $\Z G$, we get
    \[
        \rk_{G}(A) = \widetilde{\rk_N}(A) \geq \widetilde{\rk_{\FF_2 N}}(A) = \rk_{\FF_2 G}(A).
    \]
    By assumption $M$ is a finitely presented $\Z[G]$-module, and therefore
    \[
        \beta_0^G(\QQ \otimes_{\Z} M) \leq \beta_0^{\mathcal{D}_{\FF_2 G}}(\FF_2 \otimes_{\Z} M).
    \]
\end{proof}

We can now conclude our main results.

\begin{proof}[Proof of \cref{thm:L2rigid}]
    Using \cref{prop: dim_QG_F2G} and \cref{thm:DF2Grigid} we can repeat the arguments in the proof of \cite[Theorem 1.1]{Ja24}.
\end{proof}

\begin{proof}[Proof of \cref{compressedL2}]
    Assume that $H$ is $L^2$-compressed in $K$. By \cref{lem: monotone_b1}, it follows from the exact sequence
    \[
        0 \to \ ^K I_{\QQ H} \to I_{\QQ K} \to I_{\QQ}(K,H) \to 0
    \]
    that $H$ is $L^2$-independent in $K$ if and only if $\beta_1^K(I_{\QQ}(K,H))$ vanishes. Now, by \cref{thm:L2rigid}, consider $\widetilde{H}$ the $L^2$-closure of $H$ in $K$. According to \cref{lem: properties_tors}
    \[
        \beta_1^K(I_{\QQ}(K,H)) = \beta_1^K \left(^K I_{\QQ}(\widetilde{H}, H)\right) \quad \mbox{and} \quad \beta_0^K \left(^K I_{\QQ}(\widetilde{H}, H)\right) = 0.
    \]
    Hence, using again \cref{lem: monotone_b1}, it holds that
    \[
        \beta_1^K \left(^K I_{\QQ}(\widetilde{H}, H)\right) = \beta_0^K \left(^K I_{\QQ H} \right) - \beta_0^K \left(^K I_{\QQ \widetilde{H}}\right) = b_1^{(2)}(H) - b_1^{(2)}(\widetilde{H})
    \]
    where $b_1^{(2)}(\widetilde{H})$ is finite by \cref{cor: coh_tors_fg}. Thus, putting everything together we get that
    \[
        \beta_1^K(I_{\QQ}(K,H)) = b_1^{(2)}(H) - b_1^{(2)}(\widetilde{H}),
    \]
    which is at most, and hence equal to, $0$ since $H$ is $L^2$-compressed in $K$.
\end{proof}

\begin{proof}[Proof of \cref{goodsplittings}]
First consider torsion-free groups that are virtually (finitely generated free)-by-cyclic. These groups are finitely presented, have vanishing first $L^2$-Betti number \cite{Lueck1994}, and moreover, since they are locally indicable groups by \cref{lem: vFbyZ_li}, satisfy the Strong Atiyah Conjecture over $\CC$ \cite{JaikinLopez2020}. Hence, the statement follows immediately from \cref{approach} together with \cref{compressedL2}.

Now let $G=\pi_1(M)$ be the fundamental group of a $3$-manifold $M$ satisfying the stated hypotheses. These assumptions imply that $M$ is aspherical. For any finitely generated subgroup $K \leq \ker \phi$, by Scott's Core Theorem \cite{Scott1973}, there exists a compact $3$-manifold $N$ such that $K$ is isomorphic to $\pi_1(N)$. Since $K$ is of infinite index in $G$, $N$ has nonempty boundary and its interior is aspherical. Hence $K$  has cohomological dimension at most $2$. By \cite[Theorem~1.1]{KL24}, such a group $K$ is virtually free-by-cyclic. Since $G$ satisfies the Strong Atiyah Conjecture over $\CC$ \cite{KL24,FriedlLueck2019a}, $M$ is compact and $b_1^{(2)}(G)$ vanishes by \cite[Lemma 5.4(2)]{FriedlLueck2019a}, the conclusion follows again from \cref{approach} and \cref{compressedL2}.
\end{proof}

 \section{Additional applications} \label{sec:applications}

 \subsection{Computing the BNS invariant}

 Let $G$ be a finitely generated group. The \emph{Bieri--Neumann--Strebel (BNS) invariant} of $G$ is the set $\Sigma(G)$ of all characters $\phi \in \mathrm{Hom}(G, \mathbb{R}) \setminus \{0\}$ such that the monoid $G_{\phi}:=\{g\in G \mid \phi(g) \geq 0\}$ is of type $\FP_1(\Z)$, that is, the trivial $\ZZ G_{\phi}$-module $\Z$ is finitely presented. 

 The aim of this section is to give a proof of \cref{GardamKielak} which was announced by Gardam--Kielak in \cite{Oberwolfach}.

Let $\phi \colon G \to \mathbb{Z}$ be an epimorphism. For the purposes of this section, an admissible splitting dual to $\phi$ consists of the following data: an element $t \in G$ such that $\phi(t) = 1$; a finite subset $S_K \subseteq G$ such that $\phi(s) = 0$ for all $s \in S_K$; finite subsets $S_{H_1}$ and $S_{H_2}$   of elements of $\langle S_K \rangle$; and a bijection $\Phi \colon S_{H_1} \to S_{H_2}$ that induces an isomorphism $\Phi \colon \langle S_{H_1} \rangle \to \langle S_{H_2} \rangle$, such that $G = \langle S_K, t \rangle$   and $t^{-1}st = \Phi(s)$ for all $s \in S_{H_1}$. Note that $G$ splits as an HNN extension with base group $K:= \langle S_K \rangle$ and associated subgroups $ H_1:= \langle S_{H_1} \rangle$ and $H_2 := \langle S_{H_2} \rangle$, which we denote by $G = K \ast_{H_1, H_2, t}$.

\begin{lemma}\label{lem:admissible-split-list}
    Let $G$ be a group with a finite presentation $\mathcal{Q}$ and let $\phi \colon G \to \Z$ be an epimorphism. There exists an algorithm that lists the admissible splittings of $G$ dual to $\phi$. 
\end{lemma} 

\begin{proof}

By repeatedly applying Tietze transformations to $\mathcal{Q}$, we may list all the finite presentations of $G$ of the form 
\begin{equation}\label{eq:admissible-splitting-pres} \langle S_K, t \mid R \cup R' \rangle \end{equation}
where $\phi(t) = 1$, $\phi(s) = 0$ for each $s \in S_K$, and such that $R \subseteq F(S_K)$,  $R' = \{r_i \mid i \in I\} \subseteq F(S_K \cup\{t\})$, and each $r_i \in R'$ is of the form $r_i = t^{-1}u_itv_i^{-1}$ with $u_i, v_i \in F(S_K)$ for each $i \in I$.

Define $S_{H_1} := \{ u_i \mid i \in I \}$ and $S_{H_2} := \{v_i \mid i \in I \}$. Then, conjugation  by $t$ in $G$ induces an isomorphism between the subgroups they generate in $G$, namely $ \Phi \colon \langle S_{H_1} \rangle_G \to \langle S_{H_2} \rangle_G$ sends $\Phi(u_i) = v_i$ for all $i \in I$. Hence, the presentation \eqref{eq:admissible-splitting-pres} corresponds to an admissible splitting dual to $\phi$. Conversely, any admissible splitting of $G$ dual to $\phi$ gives rise to a presentation of the form \eqref{eq:admissible-splitting-pres}.
\end{proof}

\begin{lemma}[{\cite[Propositions~4.3 and 4.4]{BieriNeumannStrebel1987}}]\label{lem:BNS-alt}
    Let $G$ be a group of type $\mathrm{FP}_2(\mathbb{Q})$ and let $\phi \colon G \to \mathbb{Z}$ be an epimorphism. Then, the following are equivalent. 
    \begin{enumerate}
        \item $\phi \in \Sigma(G)$.
        \item There exists an admissible splitting $G = K \ast_{H_1, H_2, t}$ dual to $\phi$ such that $H_1 = K$.
        \item For any admissible splitting $G = K \ast_{H_1, H_2, t}$ dual to $\phi$, we have that $H_1 = K$.
    \end{enumerate}
\end{lemma}

\begin{theorem}[Feighn--Handel \cite{FeighnHandel1999}]\label{thm:Feighn-Handel} 
There exists an algorithm which takes as input a finite presentation $\mathcal{Q}$ of a finitely generated free-by-cyclic group $G$; an epimorphism $\psi \colon G \to \mathbb{Z}$ with free kernel; an element $t \in \psi^{-1}(1)$; and a finite subset $S \subseteq G$ containing $t$ such that for $H := \langle S \rangle$, $b_1^{(2)}(H)$ is known.    Let $\theta \in \mathrm{Aut}(\ker \psi)$ be the automorphism induced by the conjugation action of $t$.

The output of the algorithm is a pair of finite subsets $A, B \subset \ker \psi$ which satisfy the following properties: 
\begin{enumerate}
    \item $H = \langle A, B, t \rangle $;
    \item $\theta(A) \subset \langle A, B \rangle$;
    \item the kernel of the natural epimorphism $F(A \cup B \cup \{t\}) \to H$ is normally generated by the set $\{t^{-1}at \theta(a)^{-1} \mid a \in A\}$. 
\end{enumerate}

\end{theorem}

In particular, the output of the algorithm in \cref{thm:Feighn-Handel} is a finite presentation of $H$ corresponding to an HNN splitting with base $\langle A, B \rangle \cong F(A) \ast F(B)$ and associated subgroups $\langle A \rangle \cong F(A)$ and $\langle \theta(A) \rangle \cong F(\theta(A))$.  We call such a presentation a \emph{preferred presentation}. 

\begin{proof}[Sketch of \cref{thm:Feighn-Handel}]
    Let $R$ be the rose graph with a vertex $v_R$, and identify $\ker \psi = \pi_1(R, v_R)$. A \emph{labelled graph} $X$ is a connected graph with a basepoint $x_0$ and a map $f_X \colon (X, x_0) \to (R, v_R)$ that is injective on interiors of edges. We say that $X$ is \emph{tight} if $f_X$ is an immersion. 

A \emph{labelled graph pair} is a pair $(Z,X)$ of connected labelled graphs such that $X \subset Z$ and $X$ and $Z$ have a common basepoint $x_0$. It is \emph{tight} if $Z$ is tight. The \emph{reduced rank} of $(Z,X)$ is $\rr(Z,X) = \rank(Z) - \rank(X)$. We let $(f_{X})_{\#} \colon \pi_1(X, x_0) \to \pi_1(R, v_R)$ be the homomorphism induced by $f_{X}$, and write $W^{\#}$ to denote the subgroup $(f_{X})_{\#}(\pi_1(W, x_0)) \leq \ker \psi$, for any connected subgraph $W \subset X$ containing $x_0$.  
    
   For a finitely generated subgroup $H \leq G$, we say that $(Z,X)$ is an \emph{invariant labelled graph pair} for $H$ if $H = \langle t, X^{\#} \rangle$ and $ \langle Z^{\#} \rangle = \langle X^{\#}, \theta(X^{\#}) \rangle$. Observe that if $(Z,X)$ is an invariant labelled graph pair for $H$ then, by fixing a spanning tree for $Z$ that contains $x_0$, we may fix a basis $A$ for $X^{\#}$ and extend it to a basis $A \cup B$ of $Z^{\#}$. Since $\theta(X^{\#}) \subseteq \langle Z^{\#} \rangle$, we have that $\theta(a) \in \langle A, B \rangle$. Consider the natural quotient $\pi \colon F(A \cup B \cup \{t\}) \to H$. We have that   $\{t^{-1}at \theta(a)^{-1} \mid a \in A\} \subseteq \ker \pi$. The argument in the proof of Main Proposition in \cite[\S 6]{FeighnHandel1999}, shows that if $(Z,X)$ has minimal relative rank over all tight invariant labelled graph pairs for $H$, then the kernel of $\pi$ is normally generated by $\{t^{-1}at \theta(a)^{-1} \mid a \in A\}$. 
   
   Note that the presentation complex corresponding to the presentation 
    \[\langle A, B, t \mid \forall a \in A,  t^{-1}at\theta(a)^{-1}\rangle\]
    is aspherical. Hence, the first $L^2$-Betti number of the corresponding group is $|B| = \rr(Z,X)$. It follows that the minimal reduced rank $\rr(Z,X)$ of a tight invariant graph pair $(Z,X)$ for $H$ is equal to the first $L^2$-Betti number of $H$.

    We may now describe the algorithm for computing a preferred presentation for $H$, assuming that $b_1^{(2)}(H) =:k$ is known. First, list all tight invariant graph pairs for $H$. Equivalently, list all epimorphisms $\pi \colon F(A \cup B \cup \{t\}) \to H$ with $\pi(A), \pi(B) \subseteq \ker \psi$, $\psi(\pi(t)) = 1$, and such that $ \{t^{-1}at \theta(a)^{-1} \mid a \in A\} \subseteq \ker \pi$.
    
     Any tight invariant graph pair of relative rank $k$ (or, equivalently, any epimorphism $\pi$ with $|B| = k$), gives rise to an admissible presentation of $H$ as above. Since the tight invariant graph pairs $(Z,X)$ for $H$ with $\rr(Z,X) = k$ are exactly those of minimal relative rank, and since the set of relative ranks is a subset of non-negative integers, we are guaranteed to eventually find a tight invariant graph pair of relative rank equal to $k$ on our list.
\end{proof}

\begin{proposition}\label{prop:algo-single-epi}  
    There exists an algorithm which takes as input a finite presentation $\mathcal{Q}$ of a (finitely generated free)-by-cyclic group $G$, an epimorphism $\psi \colon G \to \Z$ with  free kernel, and an epimorphism $\phi \colon G \to \mathbb{Z}$, and outputs \texttt{yes} if $\phi \in \Sigma(G)$ and \texttt{no} otherwise. 
\end{proposition}

\begin{proof}

We begin by computing the first $L^2$-Betti number $b_1^{(2)}(\ker \phi) \in \mathbb{Z}_{\geq0}$. By the main theorem in \cite{Chen2025}, since $G$ is residually finite, $L^2$-acyclic and satisfies the Strong Atiyah Conjecture, there is an algorithm that computes the \emph{$\phi$-twisted $L^2$-Euler characteristic} $\chi^{(2)}(G; \phi) \in \mathbb{Z}$. Since $\phi$ is surjective, we have that $\chi^{(2)}(G; \phi) = \chi^{(2)}(\ker \phi)$. Because $b_i^{(2)}(\ker \phi) = 0$ for all $i \neq 1$ by \cite[Th\'eor\`eme 6.6]{Gaboriau2002}, it follows that $\chi^{(2)}(G; \phi) = - b_1^{(2)}(\ker \phi)$. Hence, we may apply Chen's algorithm to compute $b_1^{(2)}(\ker \phi) \in \mathbb{Z}_{\geq 0}$.

Next, by \cref{lem:admissible-split-list}, we may list  all the possible admissible splittings $G \cong K \ast_{H_1, H_2}$ dual to $\phi$, and simultaneously run the following procedure for each splitting. 

First, by evaluating $\psi$ on generators of $K$, we check whether $K$ is a subgroup of $\ker \psi$. If so, it follows that $K$ is a finitely generated free group. Hence, by Marshall-Hall's theorem, $H_1$ is a separable subgroup of $K$. Thus, it is possible to decide whether or not $H_1 = K$.

Suppose now that $K$ is not a subgroup of $\ker \psi$. Then, $\psi$ restricts to a non-trivial homomorphism $\psi|_K \colon K \to \mathbb{Z}$. Without loss of generality, we may assume that $\psi|_K$ is onto. We run Feighn--Handel's algorithm (\cref{thm:Feighn-Handel}) with $k := b_1^{(2)}(\ker \phi)$. For any splitting $G \cong K \ast_{H_1, H_2}$ with $b_1^{(2)}(K) = k$, the algorithm terminates by outputting a preferred presentation for $K$, in particular verifying that $b_1^{(2)}(K) = k$. Since $\chi(G) = 0$, we have that $b_1^{(2)}(H_1) = b_1^{(2)}(K)$. 

By \cref{goodsplittings}, there exists an admissible dual splitting $G \cong K \ast_{H_1, H_2}$  with $b_1^{(2)}(K) = b_1^{(2)}(H_1) = b_1^{(2)}(\ker \phi)$. Hence, by applying the Feighn--Handel procedure simultaneously to all the admissible dual splittings, the algorithm will eventually output an admissible dual splitting $G \cong K \ast_{H_1, H_2}$ and a certification that $b_1^{(2)}(K) = b_1^{(2)}(\ker \phi)$.

Once such a splitting is found, we run the algorithm from \cite{DiaoFeighn2005} to compute the Grushko decomposition of $K$, and in particular determine whether it is free. If $K$ is not free, then by \cite[Theorem 2.1]{GeogheganMihalikSapirWise2001} and \cref{lem:BNS-alt}, we have that $\phi \not\in \Sigma(G)$. If $K$ is free then once again we apply Marshall-Hall's theorem to decide whether $H_1$ is a proper subgroup of $K$. Again, by \cref{lem:BNS-alt}, this determines whether or not $\phi$ is in $\Sigma(G)$.\end{proof}

\begin{remark}
Note that for any finitely presented residually finite (but not necessarily free-by-cyclic) group $G$, the problem of deciding whether an epimorphism $\phi \colon G \to \mathbb{Z}$ is an element of $\Sigma(G)$ is \emph{semi-decidable}; that is, there exists an algorithm that terminates with the output \texttt{yes} if $\phi \in \Sigma(G)$, and otherwise it does not terminate.

Indeed, by \cref{lem:BNS-alt} we have that $\phi \in \Sigma(G)$ if and only if there exists (equivalently, for every) admissible splitting $G \cong K \ast_{H_1, H_2, \theta}$ dual to $\phi$, we have that $H_1 = K$. Thus, we may pick an admissible splitting and attempt to verify that $H_1 = K$. Since $G$ is residually finite, its word problem is solvable. Consequently, one can systematically enumerate all possible words over generators of $H_1$ and test for equality against the generators of $K$, thus eventually verifying that $H_1 = K$, if this is indeed the case. 
\end{remark}

To prove \cref{GardamKielak}, we must introduce Friedl--L\"{u}ck's \emph{$L^2$-polytope} \cite{FriedlLueck2017}. Let $G$ be an $L^2$-acyclic group of type $\mathrm{FP}$ that satisfies the Strong Atiyah Conjecture and let $H$ denote its free abelianisation. The \emph{$L^2$-polytope of $G$} is defined to be the image of the determinant of the universal $L^2$-torsion $\rho^{(2)}_u(K(G,1)) \in \mathrm{Wh}^w(G)$ under the polytope homomorphism $\mathcal{D}(G)^{\times}_{\mathrm{ab}} \to \mathcal{P}(H)$, where $\mathcal{P}(H)$ is the Grothendieck group of integral polytopes in $H_1(H; \mathbb{R})$. Note that if $\phi \in H^1(G; \QQ)$ is not trivial then $\gr_{P}(\phi) = - \alpha_{\phi}\, \chi^{(2)}(\ker \phi)$ where $P$ is the $L^2$-polytope of $G$, and $\alpha_{\phi} \in \mathbb{Q}_{\geq 0}$ is such that $\phi(G) = \alpha_{\phi} \mathbb{Z}$. We refer the reader to \cite{FriedlLueck2017, FriedlLueckTillmann2019} for further details. 

\begin{proof}[Proof of \cref{GardamKielak}]

We begin with a couple of general observations about polytopes. Let $H$ be a free abelian group of finite rank and let $V = H_1(H; \RR)$. By \cite[Lemma~3.5]{Kielak2020polytopes}, if $P_0$ and $P_1$ are polytopes in $V$ and $F$ is a face of $P_0 + P_1$, then $F$ admits a unique decomposition of the form $F = F_0 + F_1$, where $F_i$ is a face of $P_i$ for $i \in \{0,1\}$. 

Fix $\phi \in H^1(H; \RR)$ and let $P$ be a polytope in $V$. We have that 
\[ F_{\phi}(P + \overline{P}) = F_{\phi}(P) + F_{\phi}(\overline{P}).\] It follows that if $v$ is a vertex of $P$ then 
\begin{equation}\label{eq:partition} D_v(P) = \bigsqcup_{F \text{ face of } \overline{P}} D_{v + F}(P + \overline{P}).\end{equation}

Suppose now that $G$ is (finitely generated free)-by-cyclic and $H$ is the free abelianisation of $G$. By the work of Kielak \cite[Theorem 5.29]{Kielak2020polytopes}, the $L^2$-polytope of $G$ is represented by an integral polytope $P$ in $V$ which admits a marking of its vertices, such that a character $\phi \in H^1(G; \mathbb{R})$ is an element of $\Sigma(G)$ if and only if $F_{\phi}(P) = \{v\}$ where $v$ is a marked vertex. Since for any non-trivial $\phi \in H^1(G; \mathbb{Q})$ we have that $\gr_{P}(\phi) = - \alpha_{\phi} \, \chi^{(2)}(\ker \phi)$, it follows that the function $\phi \mapsto - \alpha_{\phi} \chi^{(2)}(\ker \phi)$ extends to a seminorm on $H^1(G; \RR)$, which we denote by $\| \cdot \|_T$.

Let $B = \{\phi \in H^1(G; \RR) \mid \| \phi \|_T \leq 1\}$ and let $Q = B^{*}$. By definition, $Q$ is a symmetric polytope and  $\gr_{Q}(\phi) = 2\| \phi \|_T$ for every $\phi \in H^1(G; \RR)$. Note also that the polytope $P + \overline{P}$ is symmetric and $\gr_{P + \overline{P}}(\phi) = 2\| \phi \|_T = \gr_{Q}(\phi)$. It follows by \cref{thickfunction} that $ Q = P + \overline{P}$. 

As in the proof of \cref{prop:algo-single-epi}, by \cite{Chen2025} there is an algorithm to compute $- \chi^{(2)}(\ker \phi)$ for every epimorphism $\phi \colon G \to \ZZ$. Hence, there is an algorithm to compute $\gr_{Q}$, and thus, since $G$ is finitely generated and the polytope $Q$ is integral, to compute the polytope $Q$. Note that since $P$, and thus $P + \overline{P}$, is an integral polytope, it follows that for every proper face $F$ of $P + \overline{P}$, the dual region $D_F(P + \overline{P})$ contains an epimorphism $\phi \colon G \to \ZZ$.

We claim that if $P$ is not a point then the BNS invariant of $G$ is a disjoint union of dual regions of the form $D_{v+F}(P + \overline{P})$ where $v$ is a vertex of $P$ and $F$ is a proper face of $\overline{P}$ (recall that the zero character is excluded from the BNS invariant). For each such dual region, we may pick an epimorphism $\phi \colon G \to \Z$ contained in the region and apply the algorithm from \cref{prop:algo-single-epi} to check whether or not $\phi \in \Sigma(G)$. Since there are finitely many faces of $P + \overline{P}$, the algorithm will terminate in finite time. If $P$ is a point then it  must be the case that $G \cong \mathbb{Z} \rtimes \mathbb{Z}$ and the BNS invariant is $H^1(G; \mathbb{R}) \setminus \{0\}$.

Finally, to prove the claim observe that $D_u(P) \cap D_{v}(P) = \emptyset$ for any two distinct vertices $u$ and $v$ of $P$, and since the set of connected components of $\Sigma(G)$ coincides with the collection of dual regions $\{D_v(P) \mid v \text{ marked vertex of } P\}$, it follows by \eqref{eq:partition} that for any vertex $v$ of $P$ and face $F$ of $\overline{P}$, if $ D_{v + F}(P + \overline{P}) 
\cap \Sigma(G) \neq \emptyset$ then $D_{v+F}(P + \overline{P}) \subseteq \Sigma(G)$. Moreover, if $\phi \in H^1(G;\RR)$ is an element of the BNS invariant of $G$ then $\phi \in D_v(P)$ for some vertex $v$ of $P$, and thus again by \eqref{eq:partition} there exists a face of $P + \overline{P}$ of the form $v + F$, where $F$ is a face of $\overline{P}$, and such that $\phi \in D_{v+F}(P + \overline{P})$. This proves the claim. 
\end{proof}

\subsection{Connections to the isomorphism problem}

Let $G$ be a \fbyz group. Note that for $0 \neq \phi \in H^1(G; \mathbb{Z})$, if $\ker \phi$ is finitely generated then it is free. Let $P \subseteq H_1(G; \mathbb{R})$ be a representative of the $L^2$-polytope of $G$. As in the proof of \cref{GardamKielak}, there exists an algorithm to compute $Q:= P + \overline{P}$ and the marking of the faces of $Q$ that detect the BNS invariant. Since the polytope $Q$ detects the ranks of kernels of fibered characters, we obtain the following:

\begin{lemma}\label{lem:splitting-compute}
    If $G$ is (finitely generated free)-by-cyclic then there exists an algorithm to construct the set $\mathcal{X}_n(G)$ of epimorphisms $\phi \colon G \to \mathbb{Z}$ with $\ker \phi \cong \mathbb{F}_n$ where $\mathbb{F}_n$ denotes the free group of rank $n$.
\end{lemma}

\begin{lemma}\label{lem:splitting-detect}
    There exists an algorithm to determine whether or not a (finitely generated free)-by-cyclic group splits as an HNN extension over a (finitely generated free)-by-cyclic subgroup, and compute the set of epimorphisms $\phi \colon G \to \mathbb{Z}$ corresponding to such splittings.
\end{lemma}

\begin{proof}
    We begin the proof by observing that a finitely generated subgroup $H$ of a finitely generated free-by-cyclic group $G$ is \fbyz if and only if $b_1^{(2)}(H) = 0$ (see, for instance, \cite[Theorem 3.8(4)]{FriedlLueck2019}). 

Suppose that $G$ admits an epimorphism $\phi \colon G \to \Z$ with $\gr_{Q}(\phi)= 0$. Then, $ - \chi^{(2)}(\ker \phi) = b_1^{(2)}(\ker \phi) = 0$ and thus by \cref{goodsplittings}, $G$ admits an admissible splitting dual to $\phi$ over a (finitely generated free)-by-cyclic subgroup. Conversely, any such splitting induces an epimorphism $\phi \colon G \to \Z$ with $\gr_{Q}(\phi)= 0$. Thus, $G$ splits over a \fbyz subgroup if and only if the $\mathbb{R}$-span of the vertices of the polytope $Q \subseteq H_1(G; \mathbb{R})$ is a proper subspace of $H_1(G; \mathbb{R})$. 

Arguing as in the proof of \cref{GardamKielak}, there exists an algorithm to compute $Q$ and thus to compute the set of epimorphisms $\phi \colon G \to \Z$ with $\gr_{Q}(\phi) = 0$.\end{proof}

\begin{corollary}\label{cor:conjugacy-implies-iso}
    Let $G= \mathbb{F}_n \rtimes \Z$ and suppose that $G$ does not split as an HNN extension over a \fbyz subgroup. If the conjugacy problem for $\mathrm{Out}(\mathbb{F}_n)$ has a solution then there exists an algorithm to determine whether or not a (finitely generated free)-by-cyclic group $H$ is isomorphic to $G$.
\end{corollary}

\begin{proof}
We begin by applying the algorithm from \cref{lem:splitting-detect} to check whether or not $H$ splits over a (finitely generated free)-by-cyclic subgroup. If the output of the algorithm is \texttt{yes} then clearly $G$ is not isomorphic to $H$. Suppose then that $H$ does not split. 

By \cref{lem:splitting-compute}, we may compute the sets of epimorphisms of $G$ and $H$ onto $\Z$ with finitely generated kernels of rank $n$. As in the proof of \cref{lem:splitting-detect}, since $G$ does not split over a \fbyz subgroup, the $\mathbb{R}$-span of the vertices of $Q$ generates $H_1(G;\mathbb{R})$. Hence, the set $\mathcal{X}_n(G)$ is finite for all $n \in \mathbb{N}$. The same is true for $\mathcal{X}_n(H)$.

For each character $\phi \in \mathcal{X}_n(G)$, the procedure outlined in \cref{lem:BNS-alt} computes a splitting $G \cong \mathbb{F}_n \rtimes_{\alpha} \mathbb{Z}$ dual to $\phi$. While the monodromy $\alpha \in \mathrm{Aut}(\mathbb{F}_n)$ depends on the choice of generating set and the choice of the isomorphism $\ker \phi \cong \mathbb{F}_n$, it is not hard to see that any two such choices will give rise to elements in the same conjugacy class of outer automorphisms of $\mathbb{F}_n$. 

Finally, we observe that $G \cong H$ if and only if there exist epimorphisms $\phi_1 \in \mathcal{X}_n(G)$ and $\phi_2 \in \mathcal{X}_n(H)$, with corresponding splittings $G \cong \mathbb{F}_n \rtimes_{\alpha_1} \mathbb{Z}$ and $H \cong \mathbb{F}_n \rtimes_{\alpha_2} \mathbb{Z}$, such that the outer class $[\alpha_1] \in \mathrm{Out}(\mathbb{F}_n)$ is conjugate to $[\alpha_2] \in \mathrm{Out}(\mathbb{F}_n)$. Thus, assuming there is a solution to the conjugacy problem in $\mathrm{Out}(\mathbb{F}_n)$, it is possible to determine whether or not $G$ is isomorphic to $H$.
\end{proof}

For a general free-by-cyclic group $G$, the set $\mathcal{X}_n(G)$ may be infinite. However, we observe that if there exists a subgroup $\Gamma \leq \mathrm{Out}(G)$ such that $\mathcal{X}_n(G)$ admits finitely many orbits under the action of $\Gamma$, and there exists an algorithm to compute the orbits, then the same strategy as in the proof of \cref{cor:conjugacy-implies-iso} can be used to solve the isomorphism problem, assuming a solution to the conjugacy problem in $\mathrm{Out}(\mathbb{F}_n)$. Thus, we ask the following:

\begin{problem}\label{qn:orbit-problem}
  Let $G$ be free-by-cyclic and fix $n \in \mathbb{N}$. When is it the case that $\mathcal{X}_n(G)$ has finitely many $\mathrm{Out}(G)$-orbits? Does there exist an algorithm to compute the set of orbits?
\end{problem}
 
\bibliographystyle{alpha}
\bibliography{Bibliography}

\end{document}